\documentclass[11pt]{amsart}
\usepackage{amssymb}
\usepackage{tikz}
\usepackage{youngtab}
\usepackage{upgreek}
\usepackage[T1]{fontenc}\usepackage{palatino} % better font

\usepackage[pagebackref]{hyperref} % this makes life easier for readers
\hypersetup{colorlinks,linkcolor=blue,urlcolor=blue,citecolor=blue}
\usepackage[alphabetic,abbrev,nobysame]{amsrefs} % load AFTER hyperref! 

\newcommand{\Part}[1]{
 \foreach \x [count=\s from 1] in {#1}{
 	{\ifnum\s=1
		\draw (0,\s-1)--(\x,\s-1); 
		\fi}
   \draw (0,\s) to (\x,\s);
   \foreach \y in {0, ..., \x} {\draw (\y,\s)--(\y,\s-1);}
 }}

\def\UNIT{.18} \newcommand{\PART}[1]{
\begin{tikzpicture}[xscale=\UNIT, yscale=-\UNIT] 
	\Part{#1}
\end{tikzpicture}
}
\newcommand{\link}{
\begin{tikzpicture}[scale = 0.35,thick, baseline={(0,-1ex/2)}]
  \tikzstyle{vertex} = [shape = circle, minimum size = 4pt,
    inner sep = 1pt, fill=black] 
\node[vertex] (G-1) at (0.0, 0.25) [shape = circle, draw] {}; 
\node[vertex] (G-2) at (1.5, 0.25) [shape = circle, draw] {}; 
\draw (G-1) .. controls +(0.5, -0.5) and +(-0.5, -0.5) .. (G-2); 
\end{tikzpicture} 
}

\newcommand{\ypic}{
\begin{tikzpicture}[scale = 0.35,thick, baseline={(0,-1ex/2)}] 
  \tikzstyle{vertex} = [shape = circle, minimum size = 4pt,
    inner sep = 1pt, fill=black] 
\node[vertex] (G-1) at (0.0, 0.4) [shape = circle, draw] {};
\draw[] (G-1) -- (0,-.35); 
\end{tikzpicture}
}

\newcommand{\onepic}{
\begin{tikzpicture}[scale = 0.35,thick, baseline={(0,-1ex/2)}] 
\tikzstyle{vertex} = [shape = circle, minimum size = 4pt, inner sep = 1pt] 
\node[vertex] (G--1) at (0.0, -1) [shape = circle, draw,fill=black] {}; 
\node[vertex] (G-1) at (0.0, 1) [shape = circle, draw,fill=black] {}; 
\draw (G-1) .. controls +(0, -1) and +(0, 1) .. (G--1); 
\end{tikzpicture}
}

\newcommand{\epic}{
\begin{tikzpicture}[scale = 0.35,thick, baseline={(0,-1ex/2)}]
\tikzstyle{vertex} = [shape = circle, minimum size = 4pt, inner sep = 1pt] 
\node[vertex] (G--2) at (1.5, -1) [shape = circle, draw,fill=black] {}; 
\node[vertex] (G--1) at (0.0, -1) [shape = circle, draw,fill=black] {}; 
\node[vertex] (G-1) at (0.0, 1) [shape = circle, draw,fill=black] {}; 
\node[vertex] (G-2) at (1.5, 1) [shape = circle, draw,fill=black] {}; 
\draw (G--2) .. controls +(-0.5, 0.5) and +(0.5, 0.5) .. (G--1); 
\draw (G-1) .. controls +(0.5, -0.5) and +(-0.5, -0.5) .. (G-2); 
\end{tikzpicture} 
}

\swapnumbers % theorem numbers on the LEFT! 
\newtheorem{thm}{Theorem}[section]
\newtheorem*{thm*}{Theorem}
\newtheorem{lem}[thm]{Lemma}
\newtheorem*{lem*}{Lemma}
\newtheorem{prop}[thm]{Proposition}
\newtheorem*{prop*}{Proposition}
\newtheorem{cor}[thm]{Corollary}
\newtheorem*{cor*}{Corollary}

\newtheorem*{conj*}{Conjecture}

\theoremstyle{definition}

\newtheorem*{defn*}{Definition}
\newtheorem{example}[thm]{Example}
\newtheorem*{example*}{Example}
\newtheorem{rmk}[thm]{Remark}
\newtheorem*{rmk*}{Remark}

\newtheorem*{que*}{Question}

\newcommand{\Z}{\mathbb{Z}} % integers
\newcommand{\N}{\mathbb{N}} % natural numbers
\newcommand{\UU}{\mathbf{U}} % bold U
\newcommand{\TL}{\operatorname{TL}} % Temperley-Lieb alg
\newcommand{\End}{\operatorname{End}} % endomorphisms
\newcommand{\Hom}{\operatorname{Hom}} % homomorphisms
\newcommand{\fsl}{\mathfrak{sl}} % special linear Lie alg
\newcommand{\Schur}{\mathbf{S}} % Schur algebra
\newcommand{\bil}[2]{\langle #1, #2 \rangle}
\newcommand{\qbinom}[2]{\genfrac{[}{]}{0pt}{}{#1}{#2}}
\newcommand{\M}{\mathcal{M}} % maximal vectors
\newcommand{\cc}{\mathfrak{c}} % no of paths
\newcommand{\omax}{\upmu} % orthogonal maximal vectors
\newcommand{\nat}{\upeta} % naive basis vectors
\newcommand{\PM}{\operatorname{PM}} % pairing monomial
\newcommand{\Sym}{\mathfrak{S}} % symmetric group

\renewcommand{\labelenumi}{(\alph{enumi})}
\allowdisplaybreaks

\begin{document}
\title[An orthogonal realization of representations]%
      {An orthogonal realization of representations\\of the
        Temperley--Lieb algebra}
\author{Stephen Doty}
\email{doty@math.luc.edu, tonyg@math.luc.edu}
\author{Anthony Giaquinto}
%\email{tonyg@math.luc.edu}
\dedicatory{Dedicated to the memory of Georgia Benkart}
\address{Department of Mathematics and Statistics,
  Loyola University Chicago, Chicago, IL 60660 USA}
\subjclass{Primary 16T30, 16G99, 17B37}
\keywords{Diagram algebras, Schur--Weyl duality,
 Temperley--Lieb algebras, quantized enveloping algebras, Schur algebras}
\begin{abstract}\noindent
Under a suitable hypothesis, we construct a full set of pairwise
orthogonal maximal vectors in $V^{\otimes n}$, where $V=V(1)$ is the
simple module of highest weight $1$ for the quantized enveloping
algebra $\UU(\fsl_2)$. We give a number of applications, one of which
is an orthogonal basis of the simple modules for the Temperley--Lieb
algebra $\TL_n$. We relate this new orthogonal basis to the standard
cellular basis.
\end{abstract}
\maketitle

\section*{Introduction}\noindent
We work over an arbitrary field $\Bbbk$, unless stated otherwise. Fix
a nonzero element $v$ of $\Bbbk$. Let $\Schur(n)$ be the quantum Schur
algebra (with underlying Lie algebra $\fsl_2$) over $\Bbbk$ of degree
$n$.  Let $V=V(1)$ be the two dimensional simple $\Schur(1)$-module;
its $n$th tensor power $V^{\otimes n}$ is an $\Schur(n)$-module.  The
standard nondegenerate bilinear form on $V$ extends naturally to
$V^{\otimes n}$.  Let $[k] = \sum_{t=0}^{k-1} v^{k-1-2t}$ be the
(balanced) quantum integer at $v$ corresponding to $0\le k \in \Z$.
If
\[
[n]! = [n][n-1] \cdots [1]  \ne 0 \quad\text{ in } \Bbbk
\]
then $\Schur(n)$ is split semisimple over $\Bbbk$; see
\S\ref{s:rk1SA}.  Under the same hypothesis, in \S\ref{s:orthmv} we
inductively construct a full set $\{\omax(p)\}$ of pairwise orthogonal
maximal vectors in $V^{\otimes n}$, indexed by the set of walks of
length $n$ on the Bratteli diagram.  (Lemma~\ref{l:bijections} gives a
number of equivalent ways of indexing the maximal vectors; for this
Introduction we stick to walks on Bratteli as our preferred indexing
set.)  This gives an explicit semisimple decomposition
\begin{equation}\label{e:orth-decomp}
V^{\otimes n} = \textstyle \bigoplus_p \Schur(n) \omax(p)
\end{equation}
as $\Schur(n)$-modules, where $p$ ranges over the same indexing
set. If $\omax(p)$ has weight $k$, where (necessarily) $k \equiv n$
(mod $2$) then $\Schur(n) \omax(p) \cong V(k)$, the simple
$\Schur(n)$-module of highest weight $k$. Furthermore, the
decomposition in \eqref{e:orth-decomp} is orthogonal.
This construction has a number of easy applications:
\begin{enumerate}\renewcommand{\labelenumi}{\arabic{enumi}. }
\item The linear span $C(k)$ of the maximal vectors of weight $k$ is a
  simple module for the centralizer algebra
  $\End_{\Schur(n)}(V^{\otimes n})$. If $n$ is even, we get a full set
of pairwise orthogonal $\Schur(n)$-invariants in $C(0)$.
\item The algebra $\End_{\Schur(n)}(V^{\otimes n})$ is isomorphic to the
  Temperley--Lieb algebra $\TL_n:= \TL_n(\delta)$ at parameter $\delta
  = \pm(v+v^{-1})$, and $\TL_n$ is split semisimple over $\Bbbk$.
\item $V^{\otimes n} \cong \bigoplus_{k} V(k)\otimes C(k)$, and the
  summands are simple bimodules for the commuting actions of
  $\Schur(n)$ and $\TL_n$, where the sum is over the set of $k
  \equiv n$ (mod $2$). Equivalently, the summands are simple
  $\Schur(n) \otimes \TL_n$-modules.
\item $V^{\otimes n}$ satisfies Schur--Weyl duality as such a
  bimodule; i.e., the span of the endomorphisms coming
  from each action is equal to the centralizer of the other.
\item The set $\{ C(k): k\ge 0,\ k \equiv n (\text{mod } 2) \}$ is a
  full set of simple $\TL_n$-modules, up to isomorphism.
\end{enumerate}
Using (quantum) Schur algebras $\Schur(n)$ instead of the quantized
enveloping algebra $\UU(\fsl_2)$ simplifies many proofs. Indeed, we
are able to prove all of the above from scratch, without needing any
facts about $\TL_n$ other than its dimension.  Moreover, the Schur
algebras make sense when $v=1$ (without any limiting procedure) and we
do not need to bother about integral forms or change of base ring.

The condition $[n]! \ne 0$ holds precisely when $v$ avoids certain
small roots of unity. The sufficiency of this condition for split
semisimplicity of $\TL_n = \TL_n(\delta)$ over $\Bbbk$, where $\delta=
\pm(v+v^{-1})$, is implicit in Goodman, de la Harpe, and Jones
\cite{GHJ}; see \cite{DG:PTL}*{Appendix} or \cite{DG:survey}.
Andersen, Stroppel, and Tubbenhauer \cite{AST} gave a sophisticated
proof of this fact using tilting modules, but other proofs of this
basic result are scarce in the literature.  We think that the
result ought to be better known; compare for instance with the
complicated conditions in \cite{Martin} or \cite{Benkart-Moon}.  The
proof given here is quite elementary, and the new orthogonal basis is
of independent interest. The aims of this paper are probably most
closely related to those in \cite{P-SA}, although our techniques and
results are very different.

In \S\ref{s:2nd}, we give a second construction of maximal vectors
$\nat(p)$ in $V^{\otimes n}$, also indexed by walks on the Bratteli
diagram. The $\nat(p)$ turn out (see \S\ref{s:cellular}) to realize,
in tensor space, the natural cellular basis (see \cite{GL:96}) given
by half-diagrams.  The $\nat(p)$ of weight $k$ are in a triangular
(but not unitriangular) base change relation with the $\omax(p)$ of
the same weight. In \S\ref{s:2nd}, \S\ref{s:pi} we give both recursive
and non-recursive formulas for computing the coefficients of such
transitions. Let
\[
\nat(p) = \textstyle \sum_{p'} \pi_{p,p'} \omax(p'), \qquad
\omax(p) = \textstyle \sum_{p'} \pi'_{p,p'} \nat(p')
\]
define the transition coefficients. These coefficients satisfy the
following striking properties:
\begin{enumerate}\renewcommand{\labelenumi}{\arabic{enumi}. }
  \setcounter{enumi}{5}
\item Every $\pi_{p,p'}$ is, up to sign, the reciprocal of a product of
  quantum integers.
\item There exist polynomials $\pi''_{p,p'}(t_1, t_2, \dots)$ in
  $\N[t_1,t_2,\dots]$, where $\N$ is the set of nonnegative integers,
  such that $\pi'_{p,p'} = \pi''_{p,p'}([1],[2],\dots)$.
\end{enumerate}

In \S\ref{s:action}, we give explicit formulas for the action of the
Temperley--Lieb generators on the orthogonal maximal vectors.

As $\delta=\pm(v+v^{-1})=\pm[2]$, our assumption $[n]! \ne 0$ in
particular says that $\delta\ne0$, as long as $n\ge 2$, so our
techniques yield no new information on representations of $\TL_n(0)$.
Finally, we note that all of the results of this paper make sense at
$v=1$. The condition $[n]! \ne 0$ becomes $n! \ne 0$ when $v=1$,
equivalent to the necessary and sufficient condition in Maschke's
theorem for semisimplicity of the group algebra $\Bbbk[\Sym_n]$, where
$\Sym_n$ is the symmetric group. In particular, $\TL_n(\pm2)$ is
semisimple if $n! \ne 0$ in $\Bbbk$.

Although we do not explore such applications here, we believe that our
constructions will be of some value in non-semisimple situations. We
hope to return to that question in future work.

\section{Preliminaries}\label{s:prelim}\noindent
The Temperley--Lieb algebra $\TL_n(\delta)$ appeared originally in
\cite{TL} in connection with the Potts model in mathematical physics.
It reappeared in the 1980s in the spectacular work of V.F.R.~Jones
\cites{Jones:83, Jones:85, Jones:86, Jones} on knot theory. Kauffman
\cites{K:87,K:88,K:90} (see also \cite{Kauffman}) realized it as a
diagram algebra and as a quotient algebra of the group algebra of
Artin's braid group. Birman and Wenzl \cite{Birman-Wenzl} showed that
it is isomorphic to a subalgebra of the Brauer algebra \cite{Brauer};
this also follows from Kauffman's results. See \cite{DG:survey} for a
recent survey with an historical perspective. 

Let $\Bbbk$ be a commutative ring. For any positive integer $n$ and
any element $\delta$ in $\Bbbk$, $\TL_n(\delta)$ is the unital algebra
defined by the generators $e_1, \dots, e_{n-1}$ subject to the
relations
\begin{equation}\label{e:TL}
  \begin{aligned}
  e_i^2 &= \delta e_i\\
  e_i e_j e_i = e_i \text{ if } |i-j|=1, &\quad
  e_i e_j = e_j e_i \text{ if } |i-j|>1 .
\end{aligned}
\end{equation}
The unit element $1$ of the algebra is identified with the empty
product of generators. There is an algebra isomorphism
\[
\TL_n(\delta) \cong \TL_n(-\delta)
\]
defined on generators by $e_i \mapsto -e_i$. Furthermore,
$\TL_{n-1}(\delta)$ is isomorphic to the subalgebra of $\TL_n(\delta)$
generated by $e_1, \dots, e_{n-2}$. It is well known that $\TL_n(\delta)$
is free over $\Bbbk$ of rank
\begin{equation}
  \text{rk}_\Bbbk \TL_n(\delta) = \frac{1}{n+1}\binom{2n}{n} \qquad
  \text{(the $n$th Catalan number)}.
\end{equation}

\section{The rank-$1$ quantum Schur algebra}\label{s:rk1SA}\noindent
Let $\Bbbk$ be a field throughout this section. Fix a choice of
parameter $0 \ne v \in \Bbbk$. The balanced form $[n] = [n]_v$ of the
quantum integer corresponding to a nonnegative integer $n$ is defined
by
\[
[n] = \textstyle \sum_{k=0}^{n-1} v^{-(n-1)+2k}. 
\]
In particular, $[0]=0$, $[1]=1$, and $[2] = v+v^{-1}$. The definition
is extended to negative integers by setting $[-n] = -[n]$.  Then $[n]$
makes sense when $v=1$, in which case it evaluates to the integer $n$.
If $v^2 \ne 1$ then we have $[n] = (v^n - v^{-n})/(v-v^{-1})$ for any
$n \in \Z$.  Finally, we define the quantum factorial $[n]!  = [1]
\cdots[n-1] [n]$.

For any nonnegative integer $n$ we set
\[
X(n) = \{i \in \Z \cap [-n,n]: i \equiv n \text{
  (mod~$2$)}\}.
\]
The (quantum) \emph{Schur algebra} $\Schur(n)$ is the associative
$\Bbbk$-algebra (with unit element $1$) given by the generators $E$,
$F$, and $1_i$ ($i \in X(n)$) subject to the defining relations
\begin{align*}
  1_i1_j &= \delta_{ij} 1_i, \quad \textstyle \sum_{i \in X(n)} 1_i = 1,
  \tag{R1}\label{R1}\\
  EF &- FE = \textstyle \sum_{i \in X(n)} [i]1_i , \tag{R2}\label{R2}\\
  E 1_i &=
  \begin{cases}
    1_{i+2} E &\text{if } i+2 \in X(n)\\
    0 & \text{otherwise,}
  \end{cases}
  \quad
  F 1_i =
  \begin{cases}
    1_{i-2} F &\text{if } i-2 \in X(n) \\
    0 &\text{otherwise,}
  \end{cases}
  \tag{R3}\label{R3}\\
  1_i E &=
  \begin{cases}
    E 1_{i-2} &\text{if } i-2 \in X(n)\\
    0 &\text{otherwise,}
  \end{cases}
  \quad
  1_i F =
  \begin{cases}
    F1_{i+2} &\text{if }  i+2 \in X(n)\\
    0 &\text{otherwise.}
  \end{cases}
  \tag{R4}\label{R4}
\end{align*}
Relation~\eqref{R1} says that $\{1_i: i \in X(n)\}$ is a set of
pairwise orthogonal idempotents summing to the identity. (The symbol
$\delta_{ij}$ is the usual Kronecker delta.) The algebra $\Schur(n)$
makes sense at $v = 1$, in which case the quantized integers
in~\eqref{R2} become ordinary integers.  We define an element $K$ in
$\Schur(n)$ by setting
\[
K:= \textstyle \sum_{i \in X(n)} v^i 1_i.
\]
Notice that $K$ is invertible, with $K^{-1} = \sum_{i \in X(n)} v^{-i}
1_i$. In terms of this element, relation~\eqref{R2} is equivalent to
\begin{equation}
  EF - FE = \frac{K-K^{-1}}{v-v^{-1}}
\end{equation}
provided that $v^2 \ne 1$. Although this formula is problematic at $v
= \pm 1$, the Schur algebra makes perfect sense at those values. This
is one reason we prefer to work with $\Schur(n)$ instead of the
quantized enveloping algebra $\UU(\fsl_2)$.

For convenience of notation, we introduce additional symbols $1_i = 0$
(all of which are zero) for any $i \in \Z \setminus X(n)$. With these
additional symbols, the defining relations \eqref{R1}--\eqref{R4} take
the simpler equivalent form
\begin{align*}
  1_i1_j &= \delta_{ij} 1_i, \quad\textstyle \sum_{i \in \Z} 1_i = 1
  \tag{R$1'$}\label{R1'}\\
  EF &- FE = \textstyle \sum_{i \in \Z} [i]1_i \tag{R$2'$}\label{R2'}\\
  E 1_i &= 1_{i+2} E, \quad F 1_i =  1_{i-2} F \tag{R$3'$}\label{R3'}
\end{align*}
in which all the sums are actually finite. 

\begin{lem}
  \upshape{(i)}
  There is a unique algebra involution $\omega: \Schur(n) \to
  \Schur(n)$ such that $\omega(E) = F$, $\omega(F) = E$, and
  $\omega(1_i) = 1_{-i}$, for any $i \in X(n)$.
  \upshape{(ii)}
  There is a unique algebra anti-involution $\sigma: \Schur(n) \to
  \Schur(n)$ with $\sigma(E) = E$, $\sigma(F) = F$, and $\sigma(1_i) =
  1_{-i}$ for any $i \in X(n)$.
\end{lem}

\begin{proof}
To prove (i), observe that the triple of elements
\[
(\omega(E), \omega(F), \omega(1_i)) = (F,E,1_{-i})
\]
satisfies the defining relations. As $\omega^2$ fixes the generators,
$\omega^2 = 1$. For (ii), observe that the triple of elements
\[
(\sigma(E), \sigma(F), \sigma(1_i)) = (E,F,1_{-i})
\]
satisfies the defining relations of the opposite algebra
$\Schur(n)^\text{opp}$.
\end{proof}

\begin{lem}\label{l:S-props}
The following properties hold in $\Schur(n)$.
\begin{enumerate}
\item $E^m 1_i = 1_{i+2m} E^m$ and $F^m 1_i = 1_{i-2m} F^m$, for any
  $i \in X(n)$, $m \ge 0$.
\item $E^{n+1} = F^{n+1} = 0$.
\item $E F^a = F^a E + \sum_{i \in \Z} [i] \sum_{t=0}^{a-1} F^{a-1-t}
  1_i F^t$, for any $a \ge 0$.
\item $F E^a = E^a F - \sum_{i \in \Z} [i] \sum_{t=0}^{a-1} E^{a-1-t}
  1_i E^t$, for any $a \ge 0$.
\end{enumerate}
\end{lem}

\begin{proof}
Part (a) follows from relation~\eqref{R3'} by induction on $m$, and
(b) follows from (a). One gets (c) from relation~\eqref{R2'} by
induction on $a$, and then (d) follows from (c) by applying $\omega$
and using the identity $[-i] = -[i]$.
\end{proof}

\begin{lem}\label{l:span}
  The algebra $\Schur(n)$ is linearly spanned by either of the sets
  \begin{enumerate}
  \item $\{F^a 1_i E^b: i \in X(n), 0 \le a,b \le n\}$.
  \item $\{E^a 1_i F^b: i \in X(n), 0 \le a,b \le n\}$.
  \end{enumerate}
  Hence it is finite dimensional over the field $\Bbbk$.
\end{lem}

\begin{proof}
It suffices to prove that the set in (a) spans, as the similar
property for the set (b) then follows by applying $\omega$. First
we claim that the span of the indicated elements is stable under left
multiplication by $E$, $F$, and all $1_j$. For $F$ this is obvious,
and it is equally clear for the $1_j$, as
\[
1_j(F^a 1_i E^b) = 
\begin{cases}
  F^a 1_i E^b & \text{ if } j+2a = i \\
  0 & \text{ otherwise.}
\end{cases}
\]
For left multiplication by $E$, we apply Lemma \ref{l:S-props}(c) to
obtain
\begin{align*}
E(F^a 1_i E^b) &= \left( F^a E + \sum_{i' \in \Z} [i'] \sum_{t=0}^{a-1}
F^{a-1-t} 1_{i'} F^t \right) 1_i E^b \\
 &= F^a 1_{i-2} E^{b+1} + \sum_{t=0}^{a-1} [i-2t] F^{a-1} 1_i E^b .
\end{align*}
Thus the claim is established.  The claim implies that the span is
stable under left multiplication by any element in $\Schur(n)$, hence
contains $\Schur(n) = \Schur(n) \cdot 1$. (Relation~\eqref{R1} implies
that the span contains the identity~$1$.)
\end{proof}

We will find bases for $\Schur(n)$ after studying its representations.
Let $M$ be an $\Schur(n)$-module. A vector $0 \ne m \in M$ is
\emph{maximal} if it is killed by $E$.  Any $0 \ne m \in M$ satisfying
$m = 1_i m$ ($i \in X(n)$) is a \emph{weight vector} of weight
$i$. Similarly, if $M$ is any $\Schur(n)$-bimodule then any $0 \ne m
\in M$ for which $m = 1_i m 1_j$ is a \emph{biweight vector} of
biweight $(i,j)$.

\begin{lem}\label{l:wt}
  Let $M \ne 0$ be an $\Schur(n)$-module. 
  \begin{enumerate}
  \item $M = \bigoplus_{i \in X(n)} 1_i M$ is the direct sum of its
    weight spaces.
  \item $M$ has a maximal vector, which is also a weight vector of
    some weight $i \in X(n)$.
  \item If $M$ is an $\Schur(n)$-bimodule then $M =\bigoplus_{i, j \in
    X(n)} 1_i M 1_j$.
  \end{enumerate}
\end{lem}

\begin{proof}
(a) $M = \textstyle \sum_{i \in X(n)} 1_i M$ by the second part of
  relation~\eqref{R1}. The first part of~\eqref{R1} ensures that the
  sum is direct.
  
(b) Let $m$ be any nonzero vector in the weight space $1_i M$ of
  largest possible weight $i$ (so that $1_{i'}M = 0$ for all $i' >
  i$). Then $m = 1_i m$, so $E m = E1_i m = 1_{i+2}Em = 0$.

(c) is proved similarly to (a).
\end{proof}

In particular, part (c) of the preceding lemma applies to $\Schur(n)$,
regarded as a module over itself, so we have the biweight decomposition
\begin{equation}
  \Schur(n) = \bigoplus_{i,j \in X(n)} 1_i \Schur(n) 1_j .
\end{equation}
From Lemma \ref{l:S-props}(a) it follows that for any $i \in X(n)$,
$a,b \ge 0$,
\begin{equation}\label{e:biweights}
\begin{aligned}
  F^a 1_i E^b &= 1_{i-2a} F^a 1_i E^b 1_{1-2b} = 1_{i-2a} F^a E^b
  1_{1-2b} \\ E^a 1_i F^b &= 1_{i+2a} E^a 1_i F^b 1_{1+2b} =
  1_{i+2a} E^a F^b 1_{1+2b} .
\end{aligned}
\end{equation}
This shows that $F^a 1_i E^b = 0$ unless $i-2a$, $i-2b$ both belong to
$X(n)$; similarly, $E^a 1_i F^b = 0$ unless $i+2a$, $i+2b$ both belong
to $X(n)$.

Now consider the two-sided ideal $\Schur(n) 1_n \Schur(n)$ generated
by the idempotent $1_n$ indexed by the largest element of $X(n)$.  We
have a bimodule isomorphism
\begin{equation}
  \Schur(n) 1_n \otimes 1_n \Schur(n) \cong \Schur(n) 1_n \Schur(n)
\end{equation}
given by multiplication: $x1_n \otimes 1_ny \mapsto x1_ny$. As $1_n$
is not zero and is killed by $E$, it is a maximal vector in the left
ideal $\Schur(n) 1_n$.  (Note that $1_n$ is the unique idempotent
generator with those properties.) The morphisms $\sigma$, $\omega$
commute, and we set
\[
x^* = \omega \sigma(x) = \sigma \omega(x), \quad\text{for any $x$ in
  $\Schur(n)$}.
\]
Then $*$ is an anti-involution on $\Schur(n)$. It interchanges $E,F$
and fixes all the $1_i$ ($i \in X(n)$). Moreover,
\begin{equation}
  (\Schur(n)1_n)^* = 1_n \Schur(n).
\end{equation}
In other words, $*$ interchanges the ideals $\Schur(n)1_n$,
$1_n\Schur(n)$. We define the notation
\[
V(n) := \Schur(n)1_n, \quad V(n)^* := 1_n\Schur(n).
\]
Regard these ideals as left and right $\Schur(n)$-modules, respectively.

\begin{lem}\label{l:simple}
  The set $\{x_a := F^a 1_n: 0 \le a \le n\}$ is a basis of the
  $\Schur(n)$-module $V(n)$. Thus it has dimension $n+1$. Furthermore,
  for any $0 \le a \le n$, $i \in X(n)$,
  \begin{align*}
    F x_a  =
    \begin{cases}
      x_{a+1} & \text{if } a < n\\
      0 &\text{if } a = n,
    \end{cases}
    &\qquad 
    E x_a =
    \begin{cases}
      [n-a+1][a] x_{a-1} &\text{if } a>0\\
      0 & \text{if } a=0,
    \end{cases}\\
    1_i x_a &= \delta_{i,n-2a} x_a.
  \end{align*}
  If $[n]! \ne 0$ then $V(n)$ is a simple $\Schur(n)$-module.
\end{lem}

\begin{proof}
The formulas for the action of $F$ and $1_i$ are obvious.  By 
Lemma~\ref{l:S-props}(c) we have
\[
E x_a = EF^a 1_n = (F^aE + \textstyle \sum_{j \in \Z} \sum_{t=0}^{a-1}
[j] 1_{j-2(a-1-t)} F^{a-1}) 1_n.
\]
In the double sum on the right hand side, the only nonzero terms occur
when $j - 2(a-1-t) = n-2(a-1)$, which is true if and only if $j+2t =
n$. Also, $F^aE$ acts as zero on $1_n$. Thus the above simplifies to
\[
E x_a = \textstyle \sum_{t=0}^{a-1} [n-2t] F^{a-1} 1_n = [a][n+1-a] x_{a-1}
\]
as required. The needed identity $\sum_{t=0}^{a-1} [n-2t] =
[a][n+1-a]$ is verified directly from the definition of quantum
integer. If $[n]! \ne 0$ the simplicity of $V(n)$ follows by Lemma
\ref{l:wt}. 
\end{proof}

It follows immediately that $\{x'_a := 1_n E^a : 0 \le a \le n\}$ is a
basis for the simple right ideal $V(n)^*$, satisfying analogous
formulas for the right action of the generators $E$, $F$, $1_i$.

\begin{rmk}
Lemma \ref{l:simple} may be generalized as follows. Observe that $1_n
E^j$ is a maximal vector, for any $0 \le j \le n$. Furthermore, for
any fixed $j$, the span of $\{F^i1_nE^j: 0 \le i \le n\}$ is a basis
for an $\Schur(n)$-module isomorphic to $V(n)$. The isomorphism is
given by $x_i = F^i 1_n \mapsto F^i1_nE^j$ for all $0 \le i \le n$. If
$[n]! \ne 0$ then $\Schur(n) 1_n \Schur(n)$ is the direct sum of these
submodules.
\end{rmk}

We need to introduce some additional notation. Assuming that $[n]! \ne
0$, we define the divided powers
\[
  E^{(a)} := \frac{E^a}{[a]!}, \quad F^{(a)} := \frac{F^a}{[a]!}
\]
for any integer $0 \le a \le n$.  Note that $E^{(0)} = F^{(0)} =
1$. As $E^m = F^m = 0$ for any $m > n$ it is convenient to also define
$E^{(m)} = F^{(m)} = 0$ for all $m > n$. For any $a \in \Z$ we define
the quantum binomial coefficient
\[
\qbinom{a}{n} = \frac{[a][a-1]\cdots [a-n+1]}{[n]!}.
\]
Notice that $\qbinom{a}{n} = 0$ if $0 \le a < n$ and $\qbinom{-a}{n} =
(-1)^n\qbinom{-a+n-1}{n}$ for all $a \in \Z$. At $v=1$, $\qbinom{a}{n}
= \binom{a}{n}$ is the classical binomial coefficient. Starting from
formulas in Lemma \ref{l:S-props}(c), (d) we obtain the identities
\begin{equation}
\begin{aligned}\label{e:Lu}
  E^{(a)} 1_{-i} F^{(b)} &= \sum_{t \ge 0} \qbinom{a+b-i}{t} F^{(b-t)}
  1_{-i+2(a+b-t)} E^{(a-t)} \\ F^{(b)} 1_i E^{(a)} &= \sum_{t \ge 0}
  \qbinom{a+b-i}{t} E^{(a-t)} 1_{i-2(a+b-t)} F^{(b-t)}
\end{aligned}
\end{equation}
by induction on $b$, for any $i \in X(n)$. Actually, it is enough to
prove one of them, as we get the other by applying $\omega$. (These
formulas may be compared with similar formulas in
\cite{Lusztig}*{23.1.3}.)

\begin{lem}
  If $[n]! \ne 0$ then $1_{-n} = F^{(n)} 1_n E^{(n)}$ and
  $1_n = E^{(n)} 1_{-n} F^{(n)}$. In particular, $1_{-n}$ belongs to
  $\Schur(n) 1_n \Schur(n)$ and $1_n$ belongs to $\Schur(n) 1_{-n}
  \Schur(n)$.
\end{lem}

\begin{proof}
Set $a=b=i = n$ in the second formula in \eqref{e:Lu}, to obtain the
identity
\[
F^{(n)} 1_n E^{(n)} = \sum_{t \ge 0} \qbinom{n}{t} E^{(n-t)} 1_{-3n+2t}
F^{(n-t)} .
\]
The right hand side collapses to a single nonzero term (at $t=n$),
which is equal to $1_{-n}$. This proves the first identity. The second
identity follows similarly from the first formula in \eqref{e:Lu}, or
one may simply apply $\omega$ to the first identity.
\end{proof}

By imposing the conditions $1_n = 0 = 1_{-n}$ in the defining
presentation for $\Schur(n)$, we obtain the defining presentation for
$\Schur(n-1)$. More precisely, there is a quotient map $\varphi_n:
\Schur(n) \to \Schur(n-2)$ defined by $E \mapsto E$, $F \mapsto F$,
and $1_i \mapsto 1_i$ if $|i| \le n-2$, $1_{i} \mapsto 0$ if $|i|=n$.

Let $X(n)_+ = \{m \in X(n): m \ge 0\}$. For any $m$ in $X(n)_+$, we
get (from the above) a composite algebra morphism
\begin{equation}\label{e:qout-map}
\varphi_{n-2(t-1)} \cdots \varphi_{n-2} \varphi_n: \Schur(n) \to \Schur(m)
\end{equation}
where $t \ge 0$ is defined by $n-2t = m$. If $[n]! \ne 0$ then of
course $[m]! \ne 0$ for any $2 \le m \le n$. By Lemma~\ref{l:simple},
the module $V(m)$ is simple when regarded as a $\Schur(n)$-module by
composing the $\Schur(m)$-action with the map \eqref{e:qout-map}. In
this way we have constructed a family
\[
  \{V(m) : m \in X(n)_+\}
\]
of simple $\Schur(n)$-modules. The following is proved by induction on
the elements of $X(n)_+$.

\begin{thm}\label{t:S-ss}
Assume that $0 \ne v \in \Bbbk$ satisfies $[n]! \ne 0$.
\begin{enumerate}
  \item The set $\{F^{(a)} 1_n E^{(b)} : 0 \le a,b \le n\}$ is a basis
    for the ideal $\Schur(n) 1_n \Schur(n)$.
  \item There is an exact sequence $0 \to \Schur(n) 1_n \Schur(n) \to
    \Schur(n) \to \Schur(n-2) \to 0$, that is,
    $\Schur(n)/(\Schur(n)1_n\Schur(n)) \cong \Schur(n-2)$.
  \item The algebra $\Schur(n)$ is split semisimple over $\Bbbk$, and
    $\{V(m): m \in X(n)_+\}$ is a complete set of simple
    $\Schur(n)$-modules, up to isomorphism.
  \item The set $\sqcup_{m \in X(n)_+} \{F^{(a)} 1_m E^{(b)} : 0 \le a,b
    \le m\}$ is a basis for $\Schur(n)$.
  \item $\dim_\Bbbk \Schur(n) = \sum_{m \in X(n)_+} (m+1)^2 =
    \binom{n+3}{3}$.
\end{enumerate}
\end{thm}

The algebra $\Schur(n)$ is not a bialgebra (the category of
$\Schur(n)$-modules is not closed under tensor products). However,
we have the following.

\begin{thm}\label{t:tensors}
For any nonnegative integers $r,s$ there is a unique algebra morphism
$\Delta: \Schur(r+s) \to \Schur(r) \otimes \Schur(s)$ such that
\begin{enumerate}
\item $\Delta(E) = E \otimes 1 + \sum_{i \in X(r)} v^i 1_i \otimes E =
  E \otimes 1 + K \otimes E$.
\item $\Delta(F) = F \otimes \sum_{i \in X(s)} v^{-i} 1_i + 1 \otimes
  F = F \otimes K^{-1} + 1 \otimes F$.
\item $\Delta(1_i) = \sum_{i'+i'' = i} 1_{i'} \otimes 1_{i''}$, where
  $i' \in X(r)$, $i'' \in X(s)$ in the sum. 
\end{enumerate}
\end{thm}

\begin{proof}[Proof sketch]
Verify that the elements $(\Delta(E)$, $\Delta(F)$, $\Delta(1_i)$
satisfy the defining relations \eqref{R1}--\eqref{R4} or their
equivalent form \eqref{R1'}--\eqref{R3'} for $\Schur(r+s)$. For
instance, to verify the second part of~\eqref{R1}, expand $1 \otimes
1$ in $\Schur(r) \otimes \Schur(s)$ using the second part of that
axiom for each algebra, and compare with $\sum_{i \in X(r+s)}
\sum_{i'+i'' = i} 1_{i'} \otimes 1_{i''}$.
\end{proof}

There is a unique algebra morphism $\epsilon: \Schur(n) \to \Bbbk$
such that $\epsilon(E) = \epsilon(F) = 0$, $\epsilon(1_i) =
\delta_{0,i}$. If $n$ is even, $\epsilon$ takes $K$ to $1$ (and also
$1$ to $1$) and makes $\Bbbk$ into an $\Schur(n)$-module (the trivial
module, isomorphic to $V(0)$). On the other hand, if $n$ is odd,
$\epsilon = 0$ and thus $\Schur(n)$ doesn't have a trivial module.

In the semisimple case, every finite dimensional $\Schur(n)$-module is
isomorphic to a direct sum of simple modules. This decomposition is
uniquely determined by the character of the module (the formal linear
combination of each of its weight space dimension times a placeholder
function indexed by the weight). From this it is easy to decompose
tensor products (in the semisimple case) and verify simple facts such
as the Pieri rule.

\begin{rmk}
Although we have chosen to give an elementary self-contained account
of the algebras $\Schur(n)$ in this section, the results are closely
related to the papers \cites{DG:rk1,DG:qrk1} for rank $1$ and
\cites{DG:ann,DG:IMRN,Doty:RT} more generally. The inverse limit
\[
\widehat{\UU} := \textstyle \varprojlim_{n\ge 0} \Schur(n)
\]
contains well-defined elements $E$, $F$, $K^{\pm 1}$, $1_i$ ($i \in
\Z$).  The quantized enveloping algebra $\UU = \UU_v(\fsl_2)$ is
isomorphic to the subalgebra of $\widehat{\UU}$ generated by $E$, $F$,
and $K^{\pm 1}$ and its modified form $\dot{\UU}$ (see \cite{Lusztig})
is isomorphic to the subalgebra generated by $E$, $F$ and the $1_i$
($i \in \Z$).

The Schur algebra $\Schur(n)$ is isomorphic to the quotient of
$\dot{\UU}$ by the algebra morphism $\dot{\UU} \to \Schur(n)$ taking
$E \mapsto E$, $F \mapsto F$, $1_i \mapsto 1_i$ if $i \in X(n)$ and
$1_i \mapsto 0$ otherwise. $\Schur(n)$ is also isomorphic to the
quotient of $\UU$ defined by $E \mapsto E$, $F \mapsto F$, $K \mapsto
K$. See \cite{D:invlim} for further details.
\end{rmk}

\section{Orthogonal maximal vectors}\label{s:orthmv}\noindent
Assume that $\Bbbk$ is a field and $0 \ne v$ is an element of $\Bbbk$.
In this section we construct the pairwise orthogonal maximal vectors
$\omax(p)$, under the assumption $[n]! \ne 0$. This leads to an
orthogonal basis for the simple $\TL_n(\delta)$-modules, where $\delta
= \pm(v+v^{-1})$. From this it is easy to prove semisimplicity of
$\TL_n(\delta)$.

Recall the module $V(m)$ constructed in the preceding section, with
its basis $\{x_0, x_1, \dots, x_m\}$. We change notation slightly, setting
\begin{equation}
  y_{m-2i} := x_i,  \quad \text{ for all } 0 \le i \le m. 
\end{equation}
In this notation, the basis $\{x_0, x_1, \dots, x_m\}$ becomes $\{y_k
: k \in X(m)\}$, where $y_k$ has weight $k$ for each $k$. Set
\begin{equation}
  V := V(1), \quad \text{with basis } \{y_1, y_{-1}\}.
\end{equation}
Note that $V$ is a simple $\Schur(1)$-module.  By
Theorem~\ref{t:tensors}, its $n$-fold tensor power $V^{\otimes n} = V
\otimes \cdots \otimes V$ is an $\Schur(n)$-module.  It is convenient
to introduce the set
\[
\Lambda(n) = \{(\lambda_1, \lambda_2): \lambda_1 \ge \lambda_2 \ge 0,
\lambda_1+\lambda_2 = n \}
\]
of partitions of $n$ with not more than two parts. (It is customary to
omit writing zero parts of partitions.) There is a bijection
$\Lambda(n) \to X(n)_+$ given by the rule $(\lambda_1,\lambda_2)
\mapsto \lambda_1 - \lambda_2$. We will interchangeably write
$V(\lambda)$ for the $\Schur(n)$-module that was previously denoted by
$V(m)$, where $m = \lambda_1-\lambda_2$ is an element of $X(n)_+$.

Starting from the empty partition $\emptyset$, we inductively
construct a graph, the Bratteli diagram, as follows. The nodes of the
graph are partitions of at most two parts, with the elements of
$\Lambda(l)$ in level $l$ ordered by reverse dominance. We identify
partitions with their Young diagrams as usual. Put the empty partition
at level zero. Edges are drawn between nodes whose Young diagrams
differ by exactly one box. The resulting graph is infinite; we display
its first six levels in Figure~\ref{Bratteli}.

%%%%%%%%%%%%%%%%%%%%%%%%%%%%%%%
%% BEGIN: Bratteli Diagram
%%
\begin{figure}[ht]
\begin{center}
\begin{tikzpicture}[xscale=4*\UNIT, yscale=-3*\UNIT]
%Sets the coordinates for the partitions up to level 3:
  \coordinate (0) at (0,0);
  \foreach \x in {0,...,0}{\coordinate (1\x) at (2*\x,2);}
  \foreach \x in {0,...,1}{\coordinate (2\x) at (2*\x,4);}
  \foreach \x in {0,...,1}{\coordinate (3\x) at (2*\x,6);}
  \foreach \x in {0,...,2}{\coordinate (4\x) at (2*\x,8);}
  \foreach \x in {0,...,2}{\coordinate (5\x) at (2*\x,10);}
  \foreach \x in {0,...,3}{\coordinate (6\x) at (2*\x,12);}

%edges in lattice:
  \draw (0)--(10);

  \draw (10)--(20) (10)--(21);

  \draw (20)--(30) (20)--(31) (21)--(31);

  \draw (30)--(40) (30)--(41) (31)--(41) (31)--(42);

  \draw (40)--(50) (40)--(51) (41)--(51) (41)--(52) (42)--(52);

  \draw (50)--(60) (50)--(61) (51)--(61) (51)--(62) (52)--(62) (52)--(63);
  
%partitions in lattice:
\def\NULL{\footnotesize$\emptyset$}
\begin{scope}[every node/.style={fill=white}]
  \node at (0) {\NULL};

  \node at (10) {\PART{1}};

  \node at (20) {\PART{2}};
  \node at (21) {\PART{1,1}};

  \node at (30) {\PART{3}};
  \node at (31) {\PART{2,1}};

  \node at (40) {\PART{4}};
  \node at (41) {\PART{3,1}};
  \node at (42) {\PART{2,2}};

  \node at (50) {\PART{5}};
  \node at (51) {\PART{4,1}};
  \node at (52) {\PART{3,2}};

  \node at (60) {\PART{6}};
  \node at (61) {\PART{5,1}};
  \node at (62) {\PART{4,2}};
  \node at (63) {\PART{3,3}};
\end{scope}
\end{tikzpicture}
\end{center}
\caption{Bratteli diagram up to level $6$}\label{Bratteli}
\end{figure}
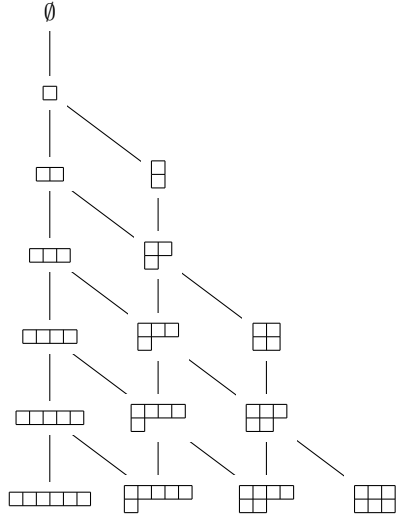
%%%%%%%%%%%%%%%%%%%%%%%%%%%%%%%%%

A \emph{walk} on the Bratteli diagram is a piecewise linear path from
$\emptyset$ to some node, consisting of vertical or diagonal line
segments respectively connecting a node on one level to the node
immediately below it or to the node immediately below and to the
right.  In such a walk, vertical (resp., diagonal) edges correspond to
adding a box to the first (resp., second) row of a partition.  For
each partition $\lambda$ with at most $2$ parts, let
$\mathcal{P}(\lambda)$ be the set of all walks to $\lambda$. The main
point of the Bratteli diagram is that it encodes the multiplicities of
the simple composition factors in the tensor powers of $V$.

\begin{lem}\label{l:path-mult}
Suppose that $[n]! \ne 0$, and let $\lambda$ be a partition of $n$
with at most two parts. The number $\cc_\lambda:=
|\mathcal{P}(\lambda)|$ of walks on the Bratteli diagram terminating
at the node $\lambda$ is equal to the multiplicity $[V^{\otimes
    n}: V(\lambda)]$ of $V(\lambda)$ in $V^{\otimes n}$.
\end{lem}

\begin{proof}
This follows by induction on $n$. The base case is trivial. Assuming
the result for $n-1$, we apply a special case of the Pieri rule to
obtain the next case. The point is that the (formal) character of
$V(\mu) \otimes V$ is obtained by summing the characters of all the
$V(\lambda)$ such that $\lambda$ is obtained from $\mu$ by adding one
box.
\end{proof}

The following is an immediate consequence of Lemma~\ref{l:path-mult}. 

\begin{cor}\label{c:path-mult}
Suppose that $[n]! \ne 0$. Then $\sum_{\lambda \in \Lambda(n)}
\cc_\lambda \dim_\Bbbk V(\lambda) = \dim_\Bbbk V^{\otimes n}$.
\end{cor}

The next result is the key observation that enables us to construct
all the maximal vectors in $V^{\otimes n}$, under suitable hypotheses.

\begin{lem}\label{l:maximal}
Suppose that $b$ is a maximal vector in $V^{\otimes (m-1)}$ of weight
$i_b$.
\begin{enumerate}
\item $\Phi_1(b) = b \otimes y_1$ is a maximal vector of weight
  $i_b+1$.
\item If $i_b>0$ then $\Phi_2(b) = [i_b] b \otimes y_{-1} - v^{i_b} Fb
  \otimes y_1$ is a maximal vector of weight $i_b-1$.
\end{enumerate}
\end{lem}

\begin{proof}
Part (a) is a special case of the fact that the tensor product of two
maximal vectors is again maximal. This follows from Theorem
\ref{t:tensors} with $(r,s) = (m-1,1)$. To get part (b), look for a
maximal vector of the form $X = b \otimes y_{-1} + \gamma Fb \otimes
y_1$ for some constant $\gamma$ yet to be determined. Applying
Theorem~\ref{t:tensors},
\begin{align*}
\Delta E (b \otimes y_{-1}) &= v^{i_b}b \otimes Ey_{-1} =
v^{i_b} b \otimes y_1 \\ \Delta E (Fb \otimes y_1) &= 
EF b \otimes y_1 =  [i_b] b \otimes y_1 .
\end{align*}
To get the second equality above, apply~\eqref{R2} to replace $EF$ by
$FE + \sum_i [i]1_i$. It follows that $\gamma = -v^{i_b}/[i_b]$ in
order to have $(\Delta E) X = 0$.  Clearing denominators gives the
formula in (b).
\end{proof}

Let $\bil{-}{-}$ be the standard (nondegenerate) symmetric bilinear
form on $V=V(1)$, defined by $\bil{y_i}{y_j} = \delta_{ij}$ extended
bilinearly. Given any multi-index $(i_1, \dots, i_{n-1}, i_n)$ in
$\{1,-1\}^n$, we set
\begin{equation*}
y_{i_1,\dots, i_{n-1},i_n} := y_{i_1} \otimes \cdots \otimes y_{i_{n-1}}
\otimes y_{i_n}.
\end{equation*}
Extend the above bilinear form to $V^{\otimes n}$, for any $n$, by
defining
\begin{equation}\label{e:bform}
\bil{y_{i_1,\dots,i_n}}{y_{j_1,\dots,j_n}} = \textstyle \prod_\alpha
\bil{y_{i_\alpha}}{y_{j_\alpha}}
\end{equation}
for any $(i_1, \dots, i_n)$, $(j_1, \dots, j_n)$ in $\{1,-1\}^n$.  As
$y_1$, $y_{-1}$ are weight vectors of weight $1$, $-1$ respectively,
it follows that $y_{i_1,\dots, i_{n-1},i_n}$ is a weight vector of
weight $\lambda_1-\lambda_2$, where $\lambda_1$ and $\lambda_2$
respectively count the number of $i_j$ equal to $1$ and $-1$. Thus the
set
\begin{equation}\label{e:ortho-basis}
  \{ y_{i_1, \dots, i_n} : (i_1, \dots, i_n) \in \{1,-1\}^n \}
\end{equation}
is a basis of weight vectors for $V^{\otimes n}$. This basis is an
orthonormal basis with respect to the bilinear form. (In particular,
weight vectors of different weights are orthogonal.)

\begin{example}\label{x:deg2}
Consider the case $n=2$ for a moment; that is, consider $V \otimes V$
with its basis $\{y_{1,1}, y_{1,-1}, y_{-1,1}, y_{-1,-1}\}$. The
vector $y_1$ is, up to a scalar, the sole maximal vector in $V$. By
Lemma~\ref{l:maximal}, the vectors
\[
\Phi_1(y_1) = y_{1,1}, \qquad \Phi_2(y_1) = y_{1,-1} - v y_{-1,1}
\]
are maximal vectors in $V \otimes V$. Notice that they are
orthogonal. We set
\[
   z_0 := \Phi_2(y_1) = y_{1,-1} - v y_{-1,1} .
\]
Observe that $z_0$ is invariant under the action of $\Schur(2)$; more
precisely, it is taken to zero by $E$, $F$ and is a weight vector of
weight $0$. Thus, its linear span is an isomorphic copy of the trivial
module in $V \otimes V$. If $v+v^{-1} = [2] \ne 0$ in $\Bbbk$, then
the $\Schur(2)$-submodule of $V\otimes V$ generated by the other
maximal vector $y_{1,1}$ has basis
\[
  \{ y_{1,1},\ Fy_{1,1}=y_{1,-1}+v^{-1} y_{-1,1}, F^2 y_{1,1} =
  (v+v^{-1})y_{-1,-1} \}
\]
and this subspace is equal to $z_0^\perp$, the orthogonal complement
with respect to the form. Thus, when $v+v^{-1} \ne 0$ we have classified
the maximal vectors in $V \otimes V$.
\end{example}

Our aim is to extend this classification to higher tensor degrees, but
first we introduce an action of $\TL_n(\delta)$ on $V^{\otimes n}$,
where $\delta = \pm(v+v^{-1})$. Let $\dot{z}_0$ be the orthogonal
projection of $V \otimes V$ onto the line $\Bbbk z_0 \subset V \otimes
V$ spanned by $z_0$.  A routine calculation shows that the matrix of
$\dot{z}_0$, with respect to the ordered basis $y_{1,1}$, $y_{1,-1}$,
$y_{-1,1}$, $y_{-1,-1}$, satisfies
\[
(v+v^{-1}) \dot{z}_0 = 
\begin{bmatrix}
  0&0&0&0\\
  0&v^{-1}&-1&0\\
  0&-1&v&0\\
  0&0&0&0
\end{bmatrix}.
\]
For each $i= 1, \dots, n-1$ we define an operator $\dot{e}_i$ on
$V^{\otimes n}$ by the relation 
\begin{equation}\label{e_i-action}
   -\dot{e}_i = 1^{\otimes(i-1)} \otimes (v+v^{-1})\dot{z}_0 \otimes
   1^{\otimes (n-1-i)}
\end{equation}
and we let $e_i$ act on $V^{\otimes n}$ as $\dot{e}_i$. The choice of
sign here is dictated by certain calculations in Section
\ref{s:cellular}; see Remark \ref{r:signs}.

\begin{lem}\label{l:tensor-action}
Letting $e_i$ act as $\dot{e}_i$ (for $i=1, \dots, n-1$) makes
$V^{\otimes n}$ into a $\TL_n(\delta)$-module, where $\delta =
\pm(v+v^{-1})$.
\end{lem}

\begin{proof}
The projection $\dot{z}_0$ satisfies $\dot{z}_0^2 = \dot{z}_0$, so for
any scalar $\delta$, the scalar multiple $\delta\dot{z}_0$ satisfies
$(\delta\dot{z}_0)^2 = \delta(\delta\dot{z}_0)$, and hence
$\dot{e}_i^2 = \delta \dot{e}_i$.  For any $\delta$, it is clear that
$\dot{e}_i \dot{e}_j = \dot{e}_j \dot{e}_i$ if $|i-j|>1$. Now set
$\delta = \pm(v+v^{-1})$. The cubic relations $\dot{e}_i \dot{e}_{i
  \pm1} \dot{e}_i = \dot{e}_i$ are verified by direct computation
(they are indifferent to the choice of sign). Thus $V^{\otimes n}$ is
a $\TL_n(\delta)$-module.
\end{proof}

\begin{lem}\label{l:actions-commute}
If $\delta = \pm(v+v^{-1})$, the above action of $\TL_n(\delta)$ on
$V^{\otimes n}$ commutes with the action of the Schur algebra
$\Schur(n)$.
\end{lem}

\begin{proof}
The $\dot{e}_i$ are multiples of idempotent projections onto
$\Schur(n)$-invariant lines, so they commute with the action of
$\Schur(n)$.
\end{proof}

\begin{lem}\label{l:adjointness}
If $b$, $b'$ are any weight vectors in $V^{\otimes (m-1)}$ of
respective weight $i_b$, $i_{b'}$ then the actions of $E$, $F$ satisfy
the following adjointness property:
\[
\bil{Eb}{b'} = v^{i_b+1} \bil{b}{Fb'}.
\]
\end{lem}

\begin{proof}
This is proved by induction, using Theorem \ref{t:tensors}, and it
obviously suffices to check it on simple tensors. Since weight vectors
of differing weights are orthogonal, both sides are zero unless $i_b+2
= i_{b'}$. In that case, the exponent $i_b+1$ is the average of $i_b$,
$i_{b'}$.
\end{proof}

\begin{lem}\label{l:orthogonality}
  If $b,b'$ are (not necessarily distinct) weight vectors in
  $V^{\otimes (m-1)}$ of respective weight $i_b$, $i_{b'}$ then
  \begin{enumerate}
  \item $\bil{\Phi_1(b)}{\Phi_1(b')}= \bil{b}{b'}$.
  \item $\bil{\Phi_1(b)}{\Phi_2(b')}=0$ if $b$ is maximal.
  \item $\bil{\Phi_2(b)}{\Phi_2(b')} = v[i_b][i_b+1] \bil{b}{b'}$ if
    $b$ is maximal and 
    $i_b>0$, $i_{b'}>0$.
  \end{enumerate}
\end{lem}

\begin{proof}
The calculation $\bil{\Phi_1b}{\Phi_1b'} = \bil{b}{b'}\bil{y_1}{y_1} =
\bil{b}{b'}$ proves part (a).

We now apply the adjointness property in Lemma~\ref{l:adjointness} to
prove part (b). First, if $i_{b'}=0$ then $\Phi_2 b' = 0$ and we are
done. Otherwise, by the definitions we have
\[
\bil{\Phi_1b}{\Phi_2b'} = [i_{b'}]\bil{b}{b'}\bil{y_{-1}}{y_1} - v^{i_{b'}}
\bil{b}{Fb'}\bil{y_1}{y_1} = -v^{i_{b'}} \bil{b}{Fb'} .
\]
By adjointness, $\bil{b}{Fb'} = v^{-i_b-1} \bil{Eb}{b'} = 0$, since
$b$ is maximal, proving part (b).

Finally, we prove part (c). First, we observe that if $i_b \ne i_{b'}$
then both sides of the formula in (c) are zero, because weight vectors
of different weight are orthogonal, as noted above. So we assume for
the rest of the proof that $i_b = i_{b'}$.  Since the cross terms are
zero, we have
\begin{align*}
\bil{\Phi_2b}{\Phi_2b'} &= [i_b]^2\bil{b}{b'}
\bil{y_{-1}}{y_{-1}} + v^{2i_b} \bil{Fb}{Fb'}
\bil{y_1}{y_1} \\
&= [i_b]^2\bil{b}{b'} + v^{2i_b} \bil{Fb}{Fb'}.
\end{align*}
Again applying adjointness, $\bil{Fb}{Fb'} = v^{-i_b+1}
\bil{EFb}{b'}$. Since $b$ is maximal, $FEb = 0$ and thus by
relation~\eqref{R2} we have $EFb = \sum [i]1_i b = [i_b]b$. Thus
$\bil{EFb}{b'} = [i_b] \bil{b}{b'}$ and
\[
\bil{Fb}{Fb'} = v^{-i_b+1}[i_b] \bil{b}{b'}.
\]
Putting this back into the above formula for $\bil{\Phi_2b}{\Phi_2b'}$
yields
\[
\bil{\Phi_2b}{\Phi_2b'} = \big([i_b]^2 + v^{i_b+1}
    [i_b]\big)\bil{b}{b'}.
\]
Now the equality $[i_b]^2 + v^{i_b+1} [i_b] = v[i_b][i_b+1]$ is
verified by an elementary calculation, completing the proof of part
(c).
\end{proof}

Now we construct maximal vectors in tensor degree $n$.  Let
$\mathcal{P}(n)$ be the set of length $n$ walks from the origin to
some node in the Bratteli diagram.  For each $p \in \mathcal{P}(n)$,
let
\[
\omax(p) = \left(\Upsilon_n\Upsilon_{n-1}\cdots \Upsilon_1\right)(1)
\]
where $\Upsilon_i = \Phi_1$ or $\Phi_2$ according to whether the
corresponding edge in $p$ is vertical or diagonal.

\begin{lem}\label{l:independence}
  For every $p \in \mathcal{P}(n)$, the vector $\omax(p)$ is
  maximal. If $[n]!  \ne 0$ then these maximal vectors are
  non-isotropic and pairwise orthogonal, hence linearly independent.
\end{lem}

\begin{proof}
The maximality follows by induction from Lemma~\ref{l:maximal}.  If
$[n]! \ne 0$, the fact that the $\omax(p)$ are pairwise orthogonal and
non-isotropic follows by induction from Lemma~\ref{l:orthogonality}.
\end{proof}

Let $\M(\lambda) = \{\omax(p) : p \in \mathcal{P}(\lambda)\}$. Any
$\omax(p)$ in $\M(\lambda)$ has weight $\lambda_1-\lambda_2$.  By
construction, we have
\begin{equation}\label{e:C-decomp}
  \M(\lambda) = \Phi_1 \M(\lambda_1-1,\lambda_2) \sqcup \Phi_2
  \M(\lambda_1,\lambda_2-1)
\end{equation}
where the first (resp., second) term in the disjoint union is omitted
if $\lambda_2 = \lambda_1$ (resp. $\lambda_2=0$). Recall the notation
$\cc_\lambda = |\mathcal{P}(\lambda)|$ from
Lemma~\ref{l:path-mult}; we have $\cc_\lambda = |\M(\lambda)|$. Write
$\cc_{\lambda_1,\lambda_2} = \cc_\lambda$ when $\lambda = (\lambda_1,
\lambda_2)$. If $[n]! \ne 0$, the $\cc_\lambda$ satisfy the recurrence
\[
\cc_{\lambda_1,\lambda_2} = \cc_{\lambda_1-1,\lambda_2} +
\cc_{\lambda_1,\lambda_2-1}
\]
where we define $\cc_\lambda = 0$ if $\lambda_1 < \lambda_2$ or
$\lambda_2 < 0$. It is clear that $\cc_{n,0} = 1$. Solving the
above recurrence with these boundary conditions gives the closed formula
\begin{equation}
  \cc_\lambda = \binom{n}{\lambda_2} - \binom{n}{\lambda_2-1}
\end{equation}
where as usual $\binom{n}{-1} = 0$.  For each $\lambda$ in
$\Lambda(n)$, we define
\[
  C(\lambda) := \text{$\Bbbk$-linear span of $\M(\lambda)$}.
\]
By Lemma~\ref{l:independence}, $\M(\lambda)$ is a basis of
$C(\lambda)$, so
\[
\dim_\Bbbk C(\lambda) = \cc_\lambda = \cc_{\lambda_1, \lambda_2} .
\]
Observe that $\cc_{n,n} = \frac{1}{n+1} \binom{2n}{n}$, the $n$th
Catalan number.

\begin{lem}\label{l:commute-Phi}
For any $1 \le i \le n-2$, the action of $e_i$ commutes
with both $\Phi_1$, $\Phi_2$ in the sense that
\[
e_i(\Phi_1 \omax(p)) = \Phi_1(e_i \omax(p)), \quad
e_i(\Phi_2 \omax(p)) = \Phi_2(e_i
\omax(p))
\]
for all walks $p$ of length $n-1$.
\end{lem}

\begin{proof}
For $\Phi_1$ this is obvious, since the action of $e_i$ is identity in
the last tensor position. Suppose that $b = \omax(p)$ for some walk $p$.
Then
\begin{align*}
  e_i(\Phi_2(b)) &= e_i([i_b] b \otimes y_{-1} - v^{i_b} Fb \otimes y_{1})\\
  &= [i_b] (e_i b) \otimes v_{-1} - v^{i_b} e_i (Fb) \otimes y_{1}\\
  &= [i_b] (e_i b) \otimes v_{-1} - v^{i_b} F (e_i b) \otimes y_{1}
\end{align*}
as $e_i$ commutes with $F$. The last line is the same as $\Phi_2(e_i
b)$, so this proves the claim for $\Phi_2$.
\end{proof}

Recall that we write $V(\lambda)$ interchangeably for
$V(\lambda_1-\lambda_2)$.  The following is the main result of this
section.

\begin{thm}\label{t:main}
  Fix $0 \ne v \in \Bbbk$. Suppose that $[n]! \ne 0$. Let $\delta =
  \pm(v+v^{-1})$, and write $\TL_n$ short for $\TL_n(\delta)$. Then:
  \begin{enumerate}
  \item For $\lambda \in \Lambda(n)$, the set $\{\omax(p) \mid p \in
    \mathcal{P}(\lambda)\}$ is a basis for the subspace of maximal
    vectors of weight $\lambda_1-\lambda_2$ in $V^{\otimes n}$.
  \item $V^{\otimes n} = \textstyle \bigoplus_{p \in \mathcal{P}(n)}
    \Schur(n) \omax(p) \cong \bigoplus_{\lambda \in \Lambda(n)} \cc_\lambda
    V(\lambda)$ is a semisimple $\Schur(n)$-module decomposition.
  \item If $p \in \mathcal{P}(n)$ has weight $k$, $\{F^a \omax(p): 0
    \le a \le k \}$ is an orthogonal basis for $\Schur(n) \omax(p)$.
    Furthermore, $\bigcup_{p \in \mathcal{P}(n)} \{F^a \omax(p): 0 \le
    a \le \operatorname{weight}(\omax(p)) \}$ is a basis of $V^{\otimes n}$,
    the elements of which are pairwise orthogonal.
  \item $C(\lambda)$ is a simple $\TL_n$-module of dimension
    $\cc_\lambda$, and $\{C(\lambda): \lambda \in \Lambda(n)\}$ is a set
    of pairwise non-isomorphic simple $\TL_n$-modules.
  \item The image of the representation $\TL_n
    \to \End_\Bbbk(V^{\otimes n})$ is equal to the centralizer algebra
    $\End_{\Schur(n)}(V^{\otimes n})$. 
  \item $\TL_n \cong \End_{\Schur(n)}(V^{\otimes n})$ is (split)
    semisimple, acts faithfully on $V^{\otimes n}$, and the set of
    simple $\TL_n$-modules in part (c) is a complete set.
  \item $V^{\otimes n}$ satisfies Schur--Weyl duality, as an
    $(\Schur(n), \TL_n)$-bimodule; i.e., the image of each action
    equals the full centralizer of the other.
  \item $V^{\otimes n} \cong \bigoplus_{\lambda \in \Lambda(n)}
      V(\lambda) \otimes C(\lambda)$ as $(\Schur(n),
      \TL_n)$-bimodules.
  \end{enumerate}
\end{thm}

\begin{proof}
The set in (a) is linearly independent, since its elements are
pairwise orthogonal. Hence $\cc_\lambda \le \dim_\Bbbk
C(\lambda)$. Thanks to the hypothesis, each maximal vector in the set
generates an isomorphic copy of $V(\lambda)$. Thus $V^{\otimes n}
\supseteq \textstyle \bigoplus_{p \in \mathcal{P}(n)} \Schur(n)
\omax(p)$. By Corollary~\ref{c:path-mult} the dimensions coincide, so
the inclusion is actually equality. This proves (b), and also the
reverse inequality in (a).
  
To prove (c), suppose that $b=\omax(p)$, $b'=\omax(p')$ are maximal
vectors of the same weight $k$. By Lemma~\ref{l:simple}, $\{F^a b: 0
\le a \le k\}$ is a basis of the simple module $\Schur(n)\omax(p)
\cong V(k)$. The basis elements all have different weights, so are
orthogonal. Furthermore, $\bil{F^a b}{F^c b'} = 0$ unless $a=c$ for
the same reason. Finally, applying adjointness
(Lemma~\ref{l:adjointness}) to $\bil{F^ab}{F^ab'}$ gives
\[
\bil{F^ab}{F^ab'} = v^{k-2a+1} \bil{EF^ab}{F^{a-1}b'}
\]
which by Lemma~\ref{l:simple} is the same as $[a][k+1-a]
\bil{F^{a-1}b}{F^{a-1}b'}$, up to a power of $v$. By induction it
follows that $\bil{F^ab}{F^ab'}$ is, up to a power of $v$, equal to
$[a]![k][k-1]\cdots[k-a+1] \bil{b}{b'}$. This completes the proof of
(c).

To prove (d), set $\lambda' = (\lambda_1-1,\lambda_2)$, $\lambda'' =
(\lambda_1,\lambda_2-1)$. By \eqref{e:C-decomp} we have the
decomposition
\[
C(\lambda) = \Phi_1 C(\lambda') \oplus \Phi_2 C(\lambda'')
\]
where the first or second factor is omitted if $\lambda'$ or
$\lambda''$ is not a partition.  By Lemma \ref{l:commute-Phi} and the
inductive hypothesis, the above is a decomposition as
$\TL_{n-1}$-modules. We note that $\Phi_1$, $\Phi_2$ induce respective
isomorphisms $C(\lambda') \cong \Phi_1 C(\lambda')$, $C(\lambda'')
\cong \Phi_2 C(\lambda'')$ as $\TL_{n-1}$-modules. By the inductive
hypothesis $C(\lambda') \ncong C(\lambda'')$ as $\TL_{n-1}$-modules.

Assume for a contradiction that there is a is a non-trivial
$\TL_n$-submodule $D$ of $C(\lambda)$. Then $0 \ne
\Hom_{\TL_n}(D,C(\lambda))$. By restriction to $\TL_{n-1}$, we have
\[
0 \ne \Hom_{\TL_{n-1}}(D,C(\lambda)).
\]
By Schur's Lemma and the fact that $C(\lambda') \ncong C(\lambda'')$,
it follows that either $D \cong \Phi_1 C(\lambda')$ or $D \cong
\Phi_2C(\lambda'')$ as $\TL_{n-1}$-modules. But this leads to a
contradiction, as the explicit formula for the action of $e_n$ shows
that neither $\Phi_1C(\lambda')$ nor $\Phi_2C(\lambda'')$ is a
$\TL_n$-module. Hence $C(\lambda)$ is a simple $\TL_n$-module, as
required. Furthermore, for two elements $\lambda \ne \mu$ of
$\Lambda(n)$, it is clear that $C(\lambda) \ncong C(\mu)$ as
$\TL_n$-modules, as they are not isomorphic on restriction to
$\TL_{n-1}$. Part (d) is proved.

Now let $Y$ be the image of the representation $\TL_n
\to \End_\Bbbk(V^{\otimes n})$, and let $Z
= \End_{\Schur(n)}(V^{\otimes n})$, the commutant of the
$\Schur(n)$-action. The commutativity of the actions of $\Schur(n)$,
$\TL_n$ (Lemma~\ref{l:actions-commute}) implies that $Y \subseteq Z$.
The semisimplicity of $\Schur(n)$ and standard double-commutant theory
implies that $Z$ is semisimple, $\{C(\lambda) \mid \lambda \in
\Lambda(n)\}$ is a complete set of pairwise non-isomorphic simple
$Z$-modules, and $\dim_\Bbbk Z = \sum_{\lambda \in \Lambda(n)}
\cc_\lambda^2$.

On the other hand, by (c) we know that the $C(\lambda)$'s are pairwise
non-isomorphic simple modules for the semisimple algebra $Y/(\text{rad
} Y)$, so
\[
\dim_\Bbbk Y \ge \dim_\Bbbk (Y/(\text{rad } Y)) \ge \textstyle
\sum_{\lambda \in \Lambda(n)} \cc_\lambda^2 = \dim_\Bbbk Z.
\]
Combining this with the opposite inequality $\dim_\Bbbk Y \le
\dim_\Bbbk Z$ coming from the inclusion $Y \subseteq Z$ proves that $Y
= Z$ (and $\text{rad } Y = 0$). In other words, $Y
= \End_{\Schur(n)}(V^{\otimes n})$ and the induced morphism from
$\TL_n$ to $\End_{\Schur(n)}(V^{\otimes n})$ is surjective, so (e) is
proved. Finally, $\dim_\Bbbk \TL_n = \sum_{\lambda \in \Lambda(n)}
\cc_\lambda^2$, by the combinatorial identity
\[
  \frac{1}{n+1}\binom{2n}{n} = \sum_{0 \le l
  \le n:\atop l \equiv n \text{ mod } 2} \left(\binom{n}{l} -
  \binom{n}{l-1}\right)^2.
\]
This shows that in fact the induced morphism $\TL_n \to Z
= \End_{\Schur(n)}(V^{\otimes n})$ is an isomorphism. In particular,
$\TL_n$ acts faithfully on $V^{\otimes n}$, so (f) is proved.  The
remaining claims (g), (h) now follow immediately by standard
arguments.
\end{proof}

\begin{rmk}
It makes sense to set $v=1$ in the above results. Then $\Schur(n)$
becomes a classical Schur algebra, and $\TL_n(\delta) =
\TL_n(\pm2)$. The condition $[n]! \ne 0$ becomes $n!  \ne 0$. In
particular, $\TL_n(\pm 2)$ is semisimple if the characteristic of
$\Bbbk$ is zero or greater than $n$.
\end{rmk}

\iffalse %%%%%%%%%%%%%%%%%%%%%%%%%%%%%%%%%%%%%%%%%%%%%%%%%%%%%%%%%%%%%
\begin{rmk}
An alternative approach to the proof of part (d) in Theorem
\ref{t:main} goes as follows. Using the same notations $Y$, $Z$ as
defined in the above argument, observe that the endomorphisms
$\dot{e}_i$ satisfy the defining relations of $\TL_n$, so there is a
surjective algebra morphism $\TL_n$ onto the subalgebra of $Z$
generated by the $\dot{e}_i$. This morphism is given by sending $e_i
\mapsto \dot{e}_i$ for all $i$. Now it suffices to show that the
$\dot{e}_i$ generate $Z$, as the morphism will then be surjective.

How can we show this? We can use a counting argument. Observe that
every diagram in $\TL_n$ induces an endomorphism in $Z$, obtained by
writing the diagram in terms of the generators and applying this map.
It suffices to show that the set of endomorphisms so constructed is
linearly independent, by dimension considerations. That is, it
suffices to know that the action of $\TL_n$ is faithful. However, I
get stuck here, as I don't see any easy argument to show the action is
faithful.
\end{rmk}
\fi %%%%%%%%%%%%%%%%%%%%%%%%%%%%%%%%%%%%%%%%%%%%%%%%%%%%%%%%%%%%%%%%%%%%

%%%%%%%%%%%%%%%%%%%%%%%%%%%%%%%%%%%%%%%%%%%%%%%%%%%%%%%%%%%%%%%%%%%%%%%

\section{A second construction of maximal vectors}\label{s:2nd}\noindent
Let $\Bbbk$ be a field.  \emph{Assume throughout this section that $0
\ne v \in \Bbbk$ is such that $[n]! \ne 0$, so that $\Schur(n)$ is
semisimple}.  Now we explore a second way to construct maximal
vectors, based on the fact that a tensor product of maximal vectors is
again maximal.  We already know that $y_1$ and $z_0 = y_{1,-1} - v
y_{-1,1}$ are maximal vectors in tensor degrees $1$ and $2$,
respectively. The vector $y_1$ is the unique (up to scalar) maximal
vector in $V = V^{\otimes 1}$, and the vectors $y_{1,1}$, $z_0$ are a
basis for the set of maximal vectors in $V^{\otimes 2}$.

If $b$ is a maximal vector in some finite dimensional
$\Schur(n)$-module then $b$ is $\Schur(n)$-invariant $\iff Fb=0 \iff
i_b=0$.  The following result constructs new invariants from existing
ones.

\begin{lem}\label{l:nesting}
  Suppose that $0 \ne v \in \Bbbk$ is such that $[n]! \ne 0$. Let $b$
  be a maximal vector in $V^{\otimes k}$, where $k \le n$.
  \begin{enumerate}
  \item $b$ is $\Schur(n)$-invariant $\iff \Phi_2(b)=0$.

  \item If $b$ is an $\Schur(n)$-invariant then so also is
  \[
  \Psi(b) := \Phi_2(y_1 \otimes b) = y_1 \otimes b \otimes y_{-1} - v
  y_{-1} \otimes b \otimes y_1.
  \]
  \end{enumerate}
\end{lem}

\begin{proof}
(a)
The linear independence of the tensors $b\otimes y_{-1}$, $Fb \otimes
y_1$ implies that $\Phi_2(b) = [i_b]b\otimes y_{-1} - v^{i_b}Fb
\otimes y_1 = 0$ if and only if $i_b=0$ and $Fb=0$.

(b) The vector $y_1 \otimes b$ has weight $1$, so by
Theorem~\ref{t:tensors} and the definition of $\Phi_2$ we have
\begin{align*}
\Phi_2(y_1\otimes b) &= y_1\otimes b \otimes y_{-1} - v F(y_1\otimes b)
\otimes y_1\\
&= y_1\otimes b \otimes y_{-1} - v y_{-1} \otimes b \otimes y_1
\end{align*}
since $Fb=0$, and the result is proved.
\end{proof}

\begin{rmk}
The vector $\Psi(b)$ in Lemma~\ref{l:nesting} is constructed from the
formula $z_0 = \Phi_2(y_1) = y_1\otimes y_{-1} - v y_{-1} \otimes y_1$
by inserting a copy of $b$ between its two tensor factors in each
term. We refer to this construction as \emph{nesting} and we say that
$b$ is nested within $\Psi(b)$.
\end{rmk}

It will turn out that (at least in the semisimple case) every
$\Schur(n)$-invariant is a linear combination of ones obtained from
invariants in lower tensor degree by either tensoring or nesting.
Furthermore, we will show that any maximal vector is a linear
combination of maximal vectors constructed by tensoring
$\Schur(n)$-invariants with copies of $y_1$, in any order.

It will be useful to have a diagrammatic calculus for depicting
maximal vectors constructed by the above process.  To start, depict
the maximal vector $1$ in $\Bbbk = V^{\otimes 0}$ by an empty symbol,
and depict the maximal vector $y_1$ in $V = V^{\otimes 1}$ by a dot
with a hanging line segment $\ypic$. Tensor products of maximal
vectors are depicted by juxtaposition. Finally, if $b$ is an invariant
then we depict the invariant $\Psi(b)$ by nesting the depiction of $b$
within a link of the form $\link$.  (In particular, $\link$ depicts
the invariant $z_0 = \Psi(1)$ in $V^{\otimes 2}$.) Any diagram
obtained by juxtaposing nestings and segments $\ypic$ in any order is
called a \emph{link diagram}.  The segments are often called
\emph{defects} in the literature.  If $d$ is a link diagram, let
$\nat(d)$ be the tensor product (in the given order) of the elements
$y_1$ corresponding to defects and invariants $\Psi(b)$ corresponding
to link nestings, constructed recursively by the process of
Lemma~\ref{l:nesting}.

\begin{example}\label{ex:link-dia}
The link diagram $d$ displayed below
\[
\begin{tikzpicture}[scale = 0.35,thick, baseline={(0,-1ex/2)}] 
  \tikzstyle{vertex} = [shape = circle, minimum size = 4pt, inner sep = 1pt,
    fill=black] 
\node[vertex] (G-1) at (0.0, 1) [shape = circle, draw] {}; 
\node[vertex] (G-2) at (1.5, 1) [shape = circle, draw] {}; 
\node[vertex] (G-3) at (3.0, 1) [shape = circle, draw] {}; 
\node[vertex] (G-4) at (4.5, 1) [shape = circle, draw] {}; 
\node[vertex] (G-5) at (6.0, 1) [shape = circle, draw] {}; 
\node[vertex] (G-6) at (7.5, 1) [shape = circle, draw] {};
\node[vertex] (G-7) at (9.0, 1) [shape = circle, draw] {};
\node[vertex] (G-8) at (10.5, 1) [shape = circle, draw] {};
\draw[] (G-8) -- (10.5,0); 
\draw[] (G-1) -- (0,0); 
\draw[] (G-2) .. controls +(1, -1) and +(-1, -1) .. (G-7); 
\draw[] (G-3) .. controls +(0.5, -0.5) and +(-0.5, -0.5) .. (G-4); 
\draw[] (G-5) .. controls +(0.5, -0.5) and +(-0.5, -0.5) .. (G-6); 
\end{tikzpicture}
\]
with $8$ vertices, $3$ links, and $2$ defects depicts the maximal
vector
\[
\nat(d)=  y_1 \otimes \Psi(z_0 \otimes z_0) \otimes y_1
\]
in tensor degree $8$.
\end{example}

A $1$-\emph{factor} is a sequence $\alpha = (\alpha_1, \dots,
\alpha_n)$ such that each $\alpha_i = \pm 1$ and the partial sums
$\alpha_1 + \cdots + \alpha_i \ge 0$ for all $i$. For each $i$ with
$\alpha_i=1$, let $j$ be the smallest index (if any) for which $i < j
\le n$ and $\alpha_i + \cdots + \alpha_j = 0$. Whenever this happens,
the indices $(i,j)$ are said to be \emph{paired}; otherwise the index
$i$ is \emph{unpaired}. Unpaired indices are also called
\emph{defects}.

\begin{lem}\label{l:bijections}
  For any $n$, $p$ such that $0 \le 2p \le n$, the following sets are
  all in bijective correspondence with one another:
  \begin{enumerate}
  \item The set $\mathcal{P}(n-p,p)$ of walks in the Bratteli diagram
    from $\emptyset$ to $(n-p,p)$.
  \item The set of link diagrams on $n$ vertices with $p$ links.
  \item The set of $1$-factors of length $n$ with $p$ pairings.
  \item The set of standard tableaux of shape $(n-p,p)$.
  \end{enumerate}
\end{lem}

\begin{proof}
The bijection between the sets in (a), (c) is obtained by matching
vertical (resp., diagonal) edges in a Bratteli walk with $1$ (resp.,
$-1$).  The bijection between the sets in (b), (c) is obtained by
labeling linked vertices in a link diagram by $1$ and $-1$ on the
left and right, resp., and labeling all defect vertices by $1$.
The bijection between the sets in (c), (d) comes from writing
the indices $i$ such that $\alpha_i=1$ (resp., $\alpha_i=-1$) into the first
(resp., second) row of a Young diagram of shape $(n-p,p)$. For
instance, the tableau
\[
\small \young(12358,467) 
\]
corresponds to the $1$-factor $\alpha=(1,1,1,-1,1,-1,-1,1)$.  Note
that the indices are entered in the tableau in order in each row from
left to right, so this always produces a standard tableau.
\end{proof}

Henceforth, we are free to index maximal vectors using any of the
equivalent indexing sets in Lemma~\ref{l:bijections}.  In other words,
if $p$, $d$, $\alpha$, $t$ are corresponding elements in the sets in parts
(a)--(d) of the lemma, then we may use any of the notations
\[
\nat(p)=\nat(d)=\nat(\alpha)=\nat(t)
\]
interchangeably for the corresponding maximal vector. The same
convention applies to the $\omax$ vectors defined in the previous
section. In the following results we use the set of $1$-factors as our
preferred indexing set.

We note the following explicit formula (which makes sense without any
hypothesis on $v$) for the expansion of a $\nat(\alpha)$ in terms of
the basis $y_{i_1,\dots,i_l}$.

\begin{lem}\label{l:N-formula}
  Suppose that $\alpha$ is a $1$-factor of length $l$.  Let
  $\mathcal{S}(\alpha)$ be the set of all sequences obtainable from
  $\alpha$ by switching the signs of any number of pairs of paired
  entries in $\alpha$. Then
  \[
  \nat(\alpha) = \sum_{i=(i_1,\dots, i_l) \in \mathcal{S}(\alpha)}
  (-v)^{\sigma(\alpha,i)} y_{i_1, \dots, i_l}
  \]
  where $\sigma(\alpha,i)$ is the minimum number of such sign
  interchanges in getting from $\alpha=(\alpha_1,\dots, \alpha_l)$ to
  $i=(i_1,\dots,i_l)$ and $y_i = y_{i_1, \dots, i_l} = y_{i_1} \otimes
  \cdots \otimes y_{i_l}$.
\end{lem}

\begin{proof}
This follows directly from the definition of $\nat(\alpha)$ and the
formula in Lemma \ref{l:nesting}.
\end{proof}

If $\alpha$ is a $1$-factor then we let $i_\alpha = \sum \alpha_j$ be
the sum of its entries. This is the same as the weight of $b =
\nat(\alpha)$, the same as the weight of $b' = \omax(\alpha)$,
and the same as the number of defects in the link diagram
corresponding to $\alpha$.  Let $||$ denote the usual juxtaposition
operation on sequences, defined by
\[
(\alpha_1,\dots,\alpha_k) || (\beta_1, \dots, \beta_l) =
(\alpha_1, \dots, \alpha_k, \beta_1, \dots, \beta_l).
\]
Furthermore, define
\[
\alpha^+ := \alpha||(1) \quad\text{and}\quad \alpha^- := \alpha||(-1).
\]
If $\alpha$ is a $1$-factor then so is $\alpha^+$, but in order to
have $\alpha^-$ also be a $1$-factor it is necessary that $i_\alpha >
0$.

If $\alpha$ is a $1$-factor with $i_\alpha > 1$, write $\alpha^{(j)}$
for the $1$-factor obtained by linking its $j$th defect with the next
one. (If the two defects in question appear in positions $j<k$ then
this linking amounts to setting $\alpha_k= -1$.)

\begin{thm}\label{t:Phi-on-N}
Let $\alpha$ be a $1$-factor of length less than $n$ with $d=i_\alpha$
defects in its corresponding link diagram, and assume that $d>0$ in
part \upshape{(b)}. Then:
\begin{enumerate}
\item $\Phi_1 \nat(\alpha) = \nat(\alpha^+)$.
\item $\Phi_2 \nat(\alpha) = \sum_{j=1}^d
  [j]\nat(\alpha^{+(j)}) = [d]\nat(\alpha^-) + \sum_{j=1}^{d-1}
  [j]\nat(\alpha^{(j)+})$.
\end{enumerate}
\end{thm}

\begin{proof}
(a) is immediate from the definition of $\nat(\alpha)$, $\nat(\alpha^+)$.

(b)
First suppose that $\alpha=(1^d)$, so that $b=\nat(1^d)=y_1^{\otimes
  d}$. The desired formula is then proved by expanding the coproduct
repeatedly using Theorem \ref{t:tensors}. We have
\[
Fb = F(y_1^{\otimes d}) = \sum_{m=0}^{d-1} y_1^{\otimes m} \otimes
y_{-1} \otimes v^{-(d-1-m)}y_1^{\otimes (d-1-m)}.
\]
Plugging this into the definition of $\Phi_2(b)$ yields
\begin{align*}
  \Phi_2(b) &= [d]b\otimes y_{-1} - v^dFb\otimes y_1 \\ &= [d]
  y_1^{\otimes d} \otimes y_{-1} - \sum_{m=0}^{d-1} v^{m+1}
  y_1^{\otimes m} \otimes y_{-1} \otimes y_1^{\otimes
    (d-m)}.\\
  \intertext{Replacing $v^{m+1}$ in the above by the right hand side of the
    (easily verified) formal identity
  $v^{m+1} = v[m+1] - [m]$ gives}
  &= [d]y_1^{\otimes d} \otimes y_{-1} - \sum_{m=0}^{d-1} (v[m+1] - [m])\,
  y_1^{\otimes m} \otimes y_{-1} \otimes y_1^{\otimes (d-m)}\\
  \intertext{and by combining the first and last terms we get}
  &= [d]\nat(1^d,-1) + [d-1] y_1^{\otimes(d-1)} \otimes y_{-1} \otimes y_1 \\
  &\hspace{1.05in} -
  \sum_{m=0}^{d-2} (v[m+1] - [m])\,
  y_1^{\otimes m} \otimes y_{-1} \otimes y_1^{\otimes (d-m)} 
\end{align*}
in which the expression after the first term on the right hand side
has the same form as the previous line, but with $d$ replaced by $d-1$.
Repeating the last simplification an appropriate number of times gives
the result.

In the general case, $b = \nat(\alpha)$ is a tensor product of $d$
copies of $y_1$ with a certain number of invariants, up to
reordering. The case of no invariant factors is already proved
above. If there is a single invariant factor then
\[
b = b_1 \otimes b_2 \otimes b_3,\quad  \text{where $b_1 = y_1^{\otimes m}$,
  $b_2$ is invariant, and $b_3 = y_1^{\otimes (d-m)}$}
\]
and where $0 \le m \le d$. Then by expanding the definition of $\Phi_2$,
\begin{align*}
  \Phi_2(b) &= [d] b \otimes y_{-1} - v^{d} Fb \otimes y_1\\ &=
      [d] b_1 \otimes b_2 \otimes b_3 \otimes y_{-1}\\ &\qquad\qquad
      - v^{d} (Fb_1 \otimes b_2 \otimes v^{-(d-m)} b_3 + b_1 \otimes
      b_2 \otimes Fb_3) \otimes y_1.
\end{align*}
This agrees exactly with the definition of $\Phi_2(b_1 \otimes b_3)$
if we insert a tensor factor of $b_2$ between the tensor positions
occupied by the inputs. This proves the first equality in part (b) in
the case $b$ has a single invariant factor. The general case is proved
similarly, by induction on the number of invariant factors in~$b$.

Finally, the last equality in (b) follows from the first by observing
that $\alpha^{+(d)} = \alpha^-$ and that $\alpha^{+(j)} =
\alpha^{(j)+}$ for all $j<d$.
\end{proof}

\begin{rmk}\label{r:special-cases}
Let $\alpha$ be a $1$-factor of length less than $n$. The following
special cases of previous results are notable:
\begin{enumerate}
\item[(i)] If $i_\alpha = 0$ then $\Phi_2 \nat(\alpha) = 0$.
\item[(ii)] If $i_\alpha = 1$ then $\Phi_2 \nat(\alpha) =
  \nat(\alpha^-)$.
\end{enumerate}
Formula (i) follows from Lemma~\ref{l:nesting}(a), and (ii)
from Theorem~\ref{t:Phi-on-N}(b).
\end{rmk}

The following result enables a recursive construction of all the
$\nat(\alpha)$, using the operators $\Phi_1$, $\Phi_2$.

\begin{cor}\label{c:recursion}
Let $\alpha$ be a $1$-factor of length less than $n$ with $d =
i_\alpha$ defects in its corresponding link diagram. Then:
\begin{enumerate}
\item $\nat(\alpha^+) = \Phi_1 \nat(\alpha)$.
\item If $d > 0$ and $[d] \ne 0$ then $\nat(\alpha^-) =
  \frac{1}{[d]} \Phi_2 \nat(\alpha) - \sum_{j=1}^{d-1}
  \frac{[j]}{[d]} \Phi_1 \nat(\alpha^{(j)})$.
\end{enumerate}
\end{cor}

\begin{proof}
Part (a) is clear. Part (b) is just a reformulation of the second
equality in Theorem~\ref{t:Phi-on-N}(b), where we have applied (a) to
replace $\nat(\alpha^{(j)+})$ by $\Phi_1 \nat(\alpha^{(j)})$ for
each $j<d$.
\end{proof}

The set of $1$-factors of a given length is partially ordered by the
lexicographical order on their corresponding sequence of partial sums.
To be precise, if $\alpha$, $\beta$ are $1$-factors of the same
length, then
\begin{equation}
\beta \le \alpha \iff \beta_1 + \cdots + \beta_j \le \alpha_1 + \cdots
+ \alpha_j
\end{equation}
for each $j$. As $\alpha_1+\cdots + \alpha_j$ is the weight (number of
defects) of the prefix sequence $(\alpha_1, \dots, \alpha_j)$, we use
the terminology \emph{weight sequence} for the sequence of partial
sums associated to a $1$-factor $\alpha$. By definition of $1$-factor,
the elements of the weight sequence are always nonnegative integers.

It is clear that the $1$-factors appearing on the right
hand side of Theorem \ref{t:Phi-on-N}(b) are linearly ordered with
respect to this ordering:
\[
\alpha^{+(d)} > \alpha^{+(d-1)} > \cdots > \alpha^{+(1)}.
\]
Equivalently, we have
\[
\alpha^- > \alpha^{(d-1)+} > \cdots > \alpha^{(1)+}.
\]
This leads to the following.

\begin{cor}\label{c:triangular}
For any $1$-factor $\alpha$ of of length at most $n$, we have
\[
\nat(\alpha) = \textstyle \sum_{\beta \le \alpha} \pi_{\alpha,\beta} \omax(\beta)
\]
where $\pi_{\alpha,\alpha} \ne 0$ and where the sum is over all
$\beta$ of the same shape as $\alpha$ satisfying the inequality. Hence
the $\nat(\alpha)$ are related to the $\omax(\beta)$ by an
invertible triangular matrix $P = (\pi_{\alpha,\beta})$ and the set
of $\nat(\alpha)$ such that $\alpha$ has shape $\lambda$ is a
$\Bbbk$-basis for the simple module $C(\lambda)$.
\end{cor}

\begin{proof}
We proceed by induction on the length of the $1$-factor. Assuming the
result
\[
\nat(\alpha) = \sum_{\beta \le \alpha} \pi_{\alpha,\beta} \omax(\beta)
\]
holds for $\alpha$, it suffices to show the similar result holds also
for $\alpha^+$ and, if $i_\alpha>0$, also for $\alpha^-$.

By Corollary
\ref{c:recursion}(a), we may apply $\Phi_1$ to both sides to obtain
\begin{align*}
\nat(\alpha^+) = \sum_{\beta \le \alpha} \pi_{\alpha,\beta}
\omax(\beta^+)
\end{align*}
where $\pi_{\alpha^+,\alpha^=} = \pi_{\alpha,\alpha} \ne 0$.

If $d=i_\alpha > 0$ then by both parts of Corollary \ref{c:recursion} and the
inductive hypothesis,
\[
\nat(\alpha^-) = \frac{1}{[d]}\sum_{\beta \le \alpha}
\pi_{\alpha,\beta} \omax(\beta^-) - \sum_{j=1}^{d-1}
\frac{[j]}{[d]} \sum_{\beta \le \alpha^{(j)}} \pi_{\alpha^{(j)},
  \beta} \omax(\beta^+)
\]
where $\pi_{\alpha^-,\alpha^-} = \frac{1}{[d]} \pi_{\alpha,\alpha} \ne 0$.
The proof is complete.
\end{proof}

We can also use the formulas in Theorem \ref{t:Phi-on-N} to recursively
compute all the basis transition coefficients in
\[
\omax(\alpha) = \sum_{\beta \le \alpha}
\pi'_{\alpha,\beta} \nat(\beta)
\]
for any given $\alpha$. This is best illustrated by an example.

\begin{example}
Let $\alpha=(1^3,-1^3)$, with $(1,2,3,2,1,0)$ the corresponding weight
sequence. (Exponents indicate repeated entries, as usual.) Then
$\omax(\alpha) = \Phi_2^3 \Phi_1^3(1)$. As $\Phi_1^3(1) = \nat(1^3)$,
we have
\begin{align*}
\Phi_2 \Phi_1^3(1) &= \Phi_2 \nat(1^3) = [3]\nat(1^3,-1) +
    [2]\nat(1^2,-1,1) + \nat(1,-1,1^2) \\
\intertext{and applying $\Phi_2$ to both sides gives}
\Phi_2^2 \Phi_1^3(1) &= [3]\big([2]\nat(1^3,-1^2) + \nat(1,-1,1,-1,1)\big) \\
& + [2]\big([2]\nat(1^2,-1,1,-1) + \nat(1^2,-1^2,1) \big)  \\
& + \big([2]\nat(1,-1,1^2,-1) + \nat(1,-1,1,-1,1) \big)
\intertext{which simplifies (after combining the second and last terms) to}
\Phi_2^2 \Phi_1^3(1) &=
    [3][2]\nat(1^3,-1^2) 
 + [2][2]\nat(1^2,-1,1,-1) + [2]\nat(1^2,-1^2,1) \\
 & + [2]\nat(1,-1,1^2,-1) + ([3]+1)\nat(1,-1,1,-1,1).
 \intertext{Finally, apply $\Phi_2$ once more
   (which by Remark \ref{r:special-cases}(ii) 
   simply appends a $-1$ to each index
   on the right in this case) to get}
\omax(\alpha) & = \Phi_2^3 \Phi_1^3(1) =
    [3][2]\nat(1^3,-1^3) 
 + [2][2]\nat(1^2,-1,1,-1^2) \\ & + [2]\nat(1^2,-1^2,1,-1) 
 + [2]\nat(1,-1,1^2,-1^2) \\ &+ ([3]+1)\nat(1,-1,1,-1,1,-1).
\end{align*}
In the above calculation, we have resisted the urge to replace $[3]+1$
by $[2]^2$, as such simplifications will not always be available.
\end{example}

We formulate the procedure of the above example in the form of yet
another corollary.

\begin{cor}\label{c:N-basis}
Given any $1$-factor $\alpha$ of length at most $n$, let $m=m(\alpha)$
be the maximum element of its corresponding weight sequence. We have
\[
\omax(\alpha) = \sum_{\beta \le \alpha} \pi'_{\alpha,\beta} \nat(\beta)
\]
where the sum is over $1$-factors of the same shape as $\alpha$, and
$\pi'_{\alpha,\alpha} \ne 0$. Furthermore, there are polynomials
$\pi''_{\alpha,\beta}(t_1,\dots, t_m)$ in $m$ commuting indeterminates
$t_1, \dots, t_m$ with nonnegative integer coefficients such that
$\pi'_{\alpha,\beta} = \pi''_{\alpha,\beta}([1],[2], \dots, [m])$. 
\end{cor}

\begin{proof}
In the base case, $\omax(1) = \nat(1)$. Assume by induction that
\[
\omax(\alpha) = \sum_{\beta \le \alpha}
\pi'_{\alpha,\beta} \nat(\beta)
\]
where $\pi'_{\alpha,\beta} = \pi''_{\alpha,\beta}([1],\dots,[m])$ and
$\pi''_{\alpha,\beta}(t_1, \dots, t_m)$ is in $\N[t_1, \dots, t_m]$.
We note that (by definition) $\Phi_1 \omax(\alpha) =
\omax(\alpha^+)$ and, if $i_\alpha>0$, $\Phi_2 \omax(\alpha) =
\omax(\alpha^-)$. Now it follows from Theorem~\ref{t:Phi-on-N} that
\[
\omax(\alpha^+) = \sum_{\beta \le \alpha}
\pi'_{\alpha,\beta} \nat(\beta^+)
\]
and, if $d(\alpha) :=i_\alpha >0$, that
\begin{align*}
\omax(\alpha^-) = \sum_{\beta \le \alpha}
\pi'_{\alpha,\beta} \Phi_2 \nat(\beta) =
\sum_{\beta \le \alpha} \pi'_{\alpha,\beta}
\sum_{j=1}^{d(\beta)} [j]\nat(\beta^{+(j)}).
\end{align*}
The result follows.
\end{proof}

\begin{rmk}
Under the same hypotheses as in the corollary, the matrices $P =
(\pi_{\alpha,\beta})$, $P' = (\pi'_{\alpha,\beta})$, indexed by pairs
$(\alpha,\beta)$ of $1$-factors of the same shape, and with rows and
columns ordered by $\le$ as defined above, satisfy the relation $PP' =
I = P'P$; that is, $P' = P^{-1}$.
\end{rmk}

Let $\gamma$, $\beta$ be $1$-factors, \emph{not necessarily} of the
same shape. We write $\gamma \subseteq \beta$ if, for every pairing in
$\gamma$, there is a pairing in $\beta$ in the same positions. For
example,
\[
(1,1,-1) \subseteq (1,1,-1,-1,1) \quad\text{and}\quad (1,1,1,-1)
\subseteq (1,-1,1,-1).
\]
We say that $\gamma$ is \emph{subordinate to} $\beta$ if and only if
$\gamma \subseteq \beta$. Notice that if $\gamma \subseteq \beta$ then
$\gamma||(1^k) \subseteq \beta||(1^l)$ for any $k,l \ge 0$.

For any $1$-factors $\alpha$, $\beta$ of the same weight and same
length $l$ (equivalently, of the same shape) an $\alpha \diamond
\beta$-\emph{sequence} is a sequence $(\gamma(0),\gamma(1),\dots,
\gamma(l))$ of $1$-factors satisfying the properties:
\begin{enumerate}\renewcommand{\labelenumi}{(\roman{enumi})}
\item $\gamma(k)$ has length $k$ for all $k$. (In particular,
  $\gamma(0) = \emptyset$.)
\item $\gamma(k) = \gamma(k-1)^+$ if $\alpha_k=1$.
\item $\gamma(k) = \gamma(k-1)^{+(j)}$ for some $j$ if $\alpha_k = -1$.
\item $\gamma(k) \subseteq \beta$ for all $k$. \label{iv}
\end{enumerate}
Thanks to the fourth condition above, any $\alpha \diamond
\beta$-sequence necessarily terminates at $\gamma(l) = \beta$, because
its terminal element has the same number of pairings as
$\beta$. Furthermore,
\begin{equation}\label{e:iff-assertion}
  \text{an $\alpha \diamond \beta$-sequence exists} \iff \beta \le
  \alpha.
\end{equation}
To see the forward implication above, observe that the second and
third conditions above ensure that $\gamma(k) \le (\alpha_1, \dots,
\alpha_k)$ for each $k$, so it follows that $\beta \le \alpha$
whenever an $\alpha \diamond \beta$-sequence exists. Conversely, if
$\beta \le \alpha$ then one can always construct an $\alpha \diamond
\beta$-sequence.

Whenever $\alpha_k = -1$, a pairing is added in $\gamma(k) =
\gamma(k-1)^{+(j)}$, by pairing the $j$th defect in $\gamma(k-1)^{+}$
with its next defect. The \emph{pairing variable} associated to this
pairing is the indeterminate $t_j$. The \emph{pairing monomial}
$\PM(\gamma,t)$ associated to an $\alpha \diamond \beta$-sequence
$\gamma = (\gamma(0),\gamma(1),\dots, \gamma(l))$ is the product of
all its pairing variables.

\begin{cor}\label{c:pi-poly}
  Let $\beta \le \alpha$ be $1$-factors of the same weight and length,
  where the length is at most $n$. Then
  \[
  \pi''_{\alpha,\beta}(t) = \textstyle \sum_\gamma \PM(\gamma,t)
  \]
  where the index $\gamma$ ranges over the set of $\alpha \diamond
  \beta$-sequences.
\end{cor}

\begin{proof}
This follows immediately from Corollary~\ref{c:N-basis}, by paying
attention only to the terms involving $\nat(\beta)$ in that expansion. 
\end{proof}

\begin{rmk}
Some basic properties of the $\pi''_{\alpha,\beta}$ are as follows:
\begin{enumerate}\renewcommand{\labelenumi}{(\roman{enumi})}
  \item If $\alpha=(1^l)$ then $\pi''_{\alpha,\alpha} = 1$ for any $l
    \le n$.
  \item $\pi''_{\alpha,\beta} \ne 0$ if and only if $\beta \le \alpha$
    (for $\alpha$, $\beta$ of the same weight and length).
  \item If $\beta$ is a $1$-factor whose pairings are all nested one
    within the next (e.g., $\beta = (1^3,-1^3,1)$) then
    $\pi''_{\alpha,\beta}$ is a monomial.
\end{enumerate}
However, it is not always a monomial.  For one example, if
$\alpha=(1^5,-1^3)$, $\beta=(1,-1,1,-1,1,-1,1^2)$ then
$\pi''_{\alpha,\beta} = t_1^3 + 2t_1^2t_3 + t_1t_3^2 + t_1^2t_5 +
t_1t_3t_5$. The authors have software to compute these polynomials
(available by request).
\end{rmk}

\section{Explicit formula for transition coefficients}\label{s:pi}%
\noindent
Let $\Bbbk$ be a field. \emph{We continue to assume that $0 \ne v \in
\Bbbk$ is such that $[n]! \ne 0$, so that $\Schur(n)$ is semisimple}.
Recall that the basis elements $\omax(\beta)$ (for $\beta$ a
$1$-factor) are pairwise orthogonal with respect to the bilinear
form. Thus we have
\begin{equation}
  \pi_{\alpha,\beta} =
  \frac{\bil{\nat(\alpha)}{\omax(\beta)}}
       {\bil{\omax(\beta)}{\omax(\beta)}}
\end{equation}
where $\pi_{\alpha,\beta}$ is the transition coefficient in the equation
\begin{equation*}
  \nat(\alpha) = \sum_{\beta \le \alpha} \pi_{\alpha,\beta}
  \omax(\beta)
\end{equation*}
from Corollary \ref{c:triangular}.  Our goal in this section is to
obtain a non-recursive description of the $\pi_{\alpha,\beta}$.

\begin{lem}\label{l:recursion}
Let $\alpha$, $\beta$ be $1$-factors of length less than $n$. Then:
\begin{enumerate}
\item $\bil{\nat(\alpha^+)}{\omax(\beta^+)} =
  \bil{\nat(\alpha)}{\omax(\beta)}$.
\item $\bil{\nat(\alpha^+)}{\omax(\beta^-)} = 0$.
\item If $i_\alpha>0$, $i_\beta>0$ then
  $\bil{\nat(\alpha^-)}{\omax(\beta^-)} = v [i_\beta+1]
  \bil{\nat(\alpha)}{\omax(\beta)}$.
\item If $d=i_\alpha > 1$ then
  $\bil{\nat(\alpha^-)}{\omax(\beta^+)} = - \sum_{j=1}^{d-1}
  \frac{[j]}{[d]} \bil{\nat(\alpha^{(j)})}{\omax(\beta)}$.
\end{enumerate}
\end{lem}

\begin{proof}
The formulas follow easily from Lemma~\ref{l:orthogonality} and
Corollary~\ref{c:recursion}. More precisely, we get (a) from
Corollary~\ref{c:recursion}(a) and Lemma~\ref{l:orthogonality}(a). Part
(b) is immediate from Lemma~\ref{l:orthogonality}(b), by setting
$b=\nat(\alpha^+)$ and $b' = \omax(\beta^-)$ there. We now prove
part (c). By Corollary~\ref{c:recursion}(b) and the definition of
$\omax(\beta^-)$ we have
\begin{align*}
  \bil{\nat(\alpha^-)}{\omax(\beta^-)} &= \textstyle
  \bil{\frac{1}{[d]} \Phi_2 \nat(\alpha) - \sum_{j=1}^{d-1}
    \frac{[j]}{[d]} \Phi_1 \nat(\alpha^{(j)})}{\Phi_2
    \omax(\beta)} \\
  &= \textstyle\frac{1}{[d]}\bil{\Phi_2\nat(\alpha)}
      {\Phi_2\omax(\beta)} \quad\text{(by part (b))} \\
      &= \textstyle\frac{1}{[d]}v[d][i_\beta+1] \bil{\nat(\alpha)}
      {\omax(\beta)} \quad\text{(by Lemma~\ref{l:orthogonality}(c))}
\end{align*}
and part (c) is proved once we cancel the factors of $[d]$. Finally,
we prove part (d). Again by the Corollary~\ref{c:recursion}(b) we have
\begin{align*}
  \bil{\nat(\alpha^-)}{\omax(\beta^+)} &= \textstyle
  \bil{\frac{1}{[d]} \Phi_2 \nat(\alpha) - \sum_{j=1}^{d-1}
    \frac{[j]}{[d]} \Phi_1 \nat(\alpha^{(j)})}{\Phi_1
    \omax(\beta)} \\ &= \textstyle - \sum_{j=1}^{d-1}
  \frac{[j]}{[d]} \bil{\Phi_1 \nat(\alpha^{(j)})}{\Phi_1
    \omax(\beta)} \quad\text{(by Lemma \ref{l:orthogonality}(b))} \\
  &= \textstyle - \sum_{j=1}^{d-1}
  \frac{[j]}{[d]} \bil{\nat(\alpha^{(j)})}{\omax(\beta)}
\end{align*}
where we have again applied Lemma~\ref{l:orthogonality}(a). The proof
is complete.
\end{proof}

\begin{rmk}
In order for $\alpha^-$ and $\beta^+$ to be $1$-factors
satisfying $i_{\alpha^-} = i_{\beta^+}$, it is necessary and
sufficient that $i_\alpha > 1$ and $i_\beta = i_\alpha - 2$. This
explains the hypothesis on part (d) of the last result.
\end{rmk}

Let $\alpha$ be a $1$-factor. We say that a $1$-factor $\beta$ is
$\alpha$-\emph{compatible} if for every pairing in $\alpha$, the
corresponding entries of $\beta$ have opposite signs.

\begin{example}
Let $\alpha = (1^m,-1)$. The only $1$-factors $\beta$ which are
$\alpha$-compatible are $\beta=\alpha$ and $\beta=(1^{m-1},-1,1)$.
The $1$-factor $(1^2,-1^2,1)$ is $(1^3,-1^2)$-incompatible and
$(1,-1,1,-1,1^2)$ is $(1^4,-1^2)$-incompatible. The last example
illustrates the situation addressed by the next result.
\end{example}

\begin{lem}\label{l:defect}
Let $\alpha$, $\beta$ be $1$-factors of the same weight and length
at most $n$. If $\alpha$ has a defect in some position and the
corresponding entry of $\beta$ is $-1$, then $\beta$ must be
$\alpha$-incompatible.
\end{lem}

\begin{proof}
This follows by a simple counting argument. Assume for a contradiction
that $\beta$ is $\alpha$-compatible. We may write
\[
\alpha = \alpha' || (1) || \alpha'' \quad\text{and similarly} \quad
\beta = \beta' || (-1) || \beta''
\]
where the $1$ between $\alpha'$ and $\alpha''$ is a defect and the
lengths of $\alpha'$, $\beta'$ coincide (which forces the same to be
true of $\alpha''$, $\beta''$). This implies that $\alpha'$,
$\alpha''$ are $1$-factors and all pairings in $\alpha$ occur either
in $\alpha'$ or $\alpha''$. (We do not claim that $\beta'$ or
$\beta''$ are $1$-factors.) By the compatibility assumption, there are
at least as many $-1$ entries in $\beta' || \beta''$ as there are in
$\alpha' || \alpha''$. But this implies that $\beta$ has at least one
more $-1$ entry than $\alpha$. This means that $\beta$ must have a
different weight than $\alpha$, so we have reached the desired
contradiction.
\end{proof}

Let $\beta$ be a $1$-factor of length at most $n$ with $m$ pairings,
and let $i_1, \dots, i_m$ be the positions of its $-1$ entries. Set
$s_j = s_j(\beta)$ equal to $\beta_1 + \cdots + \beta_{i_j - 1}$, for
each $j = 1, \dots, m$. It follows from Lemma \ref{l:orthogonality}
that
\begin{equation}\label{e:length-squared}
  \bil{\nat(\beta)}{\nat(\beta)} = v^m \prod_{j=1}^m [s_j][s_j+1].
\end{equation}
If $\alpha$ is another $1$-factor with $m$ pairings, we define $s'_j =
s'_j(\alpha,\beta)$ by
\[
s'_j = s'_j(\alpha,\beta) = 
\begin{cases}
  -s_j & \text{ if } \alpha_{i_j} = 1 \\
  s_j+1 & \text{ if } \alpha_{i_j} = -1.
\end{cases}
\]
Then we have the following result, which is the main result of this
section.

\begin{thm}\label{t:closed-formula}
Let $\alpha$, $\beta$ be $1$-factors of the same length, which is at
most $n$, each with $m$ pairings. Then
\[
\bil{\nat(\alpha)}{\omax(\beta)} =
\begin{cases}
  v^m \prod_{j=1}^m [s'_j]  & \text{ if $\beta$ is $\alpha$-compatible,}\\
  \qquad 0 & \text{ otherwise.}
\end{cases}
\]
\end{thm}

\begin{proof}
The proof is by induction on length. There are four cases to consider,
in parallel with the four parts of Lemma~\ref{l:recursion}.
The base cases are trivial.

\textbf{Case 1.}  For any $1$-factors $\alpha$, $\beta$ of the same
shape, $\beta^+$ is $\alpha^+$-compatible if and only if $\beta$ is
$\alpha$-compatible. In the compatible situation, the desired formula
for $\bil{\nat(\alpha^+)}{\omax(\beta^+)}$ follows from
Lemma~\ref{l:recursion}(a) and the inductive hypothesis. In the
incompatible situation, $\bil{\nat(\alpha^+)}{\omax(\beta^+)} =
0$ by the inductive hypothesis.

\textbf{Case 2.} For any $1$-factors $\alpha$, $\beta$ such that
$\alpha^+$, $\beta-$ have the same shape,
$\bil{\nat(\alpha^+)}{\omax(\beta^-)} = 0$ by
Lemma~\ref{l:recursion}(b), and $\beta^-$ is always
$\alpha^+$-incompatible, by Lemma \ref{l:defect}.

\textbf{Case 3.} If $i_\alpha>0$, $i_\beta>0$ then for any $1$-factors
$\alpha$, $\beta$ of the same shape, $\beta^-$ is
$\alpha^-$-compatible if and only if $\beta$ is $\alpha$-compatible.
In the compatible situation, the desired formula
for $\bil{\nat(\alpha^-)}{\omax(\beta^-)}$ follows from
Lemma~\ref{l:recursion}(c) and the inductive hypothesis. In the
incompatible situation, $\bil{\nat(\alpha^+)}{\omax(\beta^+)} = 0$
by the inductive hypothesis.

\textbf{Case 4.}  Assume that $\beta^+$ is $\alpha^-$-compatible,
where $\beta^+$, $\alpha^-$ have the same shape.  We wish to compute
$\bil{\nat(\alpha^-)}{\omax(\beta^+)}$.  Write $\alpha = \alpha'
|| (1) || \alpha''$, where the $1$ between $\alpha'$, $\alpha''$ is
the rightmost defect in $\alpha$. The weight of $\alpha''$ is
zero. Hence
\[
\alpha^- = \alpha' || (1) || \alpha'' || (-1) \quad\text{and} \quad
\beta^+ = \beta' || (-1) || \beta'' || (1)
\]
where the lengths of $\alpha'$, $\beta'$ agree (and similarly for the
lengths of $\alpha''$, $\beta''$). The $-1$ between $\beta'$,
$\beta''$ is forced by the compatibility assumption. That assumption
also forces the weight of $\beta''$ to be zero. (But $\beta''$ need
not be a $1$-factor.) Now $\beta = \beta' || (-1) ||
\beta''$. Furthermore, the weights of $\alpha'$, $\beta'$ must
agree. This means that all defects in $\alpha^-$ are contained within
the sub $1$-factor $\alpha'$. Now we have
\[
\bil{\nat(\alpha^-)}{\omax(\beta^+)} = - \sum_{j=1}^{d-1}
\frac{[j]}{[d]} \bil{\nat(\alpha^{(j)})}{\omax(\beta)}
\]
by Lemma \ref{l:recursion}(d). Now $\alpha^{(d-1)} = \alpha' ||
(-1) || \alpha''$ is obtained by flipping the sign of the last defect
in $\alpha$, and $\beta$ is $\alpha^{(d-1)}$-compatible. On the other
hand, $\beta$ is never $\alpha^{(j)}$-compatible for any $j<d-1$,
thanks to Lemma~\ref{l:defect}. Hence the above sum collapses to a
single term and yields the equality
\[
\bil{\nat(\alpha^-)}{\omax(\beta^+)} =
\frac{-[d-1]}{[d]}\bil{\nat(\alpha^{(d-1)})}{\omax(\beta)}
\]
Since $\alpha^{(d-1)}$ and $\beta$ have shorter length than $\alpha^-$
and $\beta^+$, we have
\begin{equation*}
  \bil{\nat(\alpha^-)}{\omax(\beta^+)} = \frac{-v^m[d-1]}{[d]}\prod_{j=1}^m
    [s_j'(\alpha^{(d-1)},\beta)].\label{reduction}
\end{equation*}
In the last formula $m$ is the number of pairings in $\beta$ which is
the same as the number of pairings in $\beta^+$. Moreover, the
positions $i_1,\ldots, i_m$ of the $-1$ entries in $\beta$ and
$\beta^+$ are the same. Thus $s_j(\beta)=s_j(\beta^+)$ for all
$j$. Suppose that the $-1$ entry in $\beta = \beta'||(-1)||\beta''$
between $\beta'$ and $\beta''$ is in position $i_r$. Then
$\alpha^{(d-1)}_{i_t}= \alpha^-_{i_t}$ for all $t\neq r$, while
$\alpha^{(d-1)}_{i_r}=-1$ and $\alpha^-_{i_r}=1$. Furthermore,
$s_{r}(\beta^+)=s_r(\beta)=d-1$, so we have
$s'_r(\alpha^{(d-1)},\beta)=d$ and
$s'_r(\alpha^-,\beta^+)=-(d-1)$. Thus the denominator in the above
product cancels with $[s'_r(\alpha^{(d-1)},\beta)] = [d]$, leaving
behind a factor of $[-(d-1)] = [s'_r(\alpha^-,\beta^+)]$, while for
all $j \ne r$ we have $[s'_j(\alpha^-,\beta^+)]=
[s'_j(\alpha^{(d-1)},\beta)]$. Hence
\[
\bil{\nat(\alpha^-)}{\omax(\beta^+)} = v^m\prod_{j=1}^m
    [s_j'(\alpha^-,\beta^+)]
\]
as required.

It remains only to deal with the incompatibility situation for
Case~4. Assume from now on that $\beta^+$ is $\alpha^-$-incompatible,
where $\alpha^-$, $\beta^+$ have the same shape. We need to argue that
$\bil{\nat(\alpha^-)}{\omax(\beta^+)} = 0$, by showing that the
right hand side of Lemma \ref{l:recursion}(d) evaluates to zero.  As
above, write
\[
\alpha^- = \alpha' || (1) || \alpha'' || (-1)
\]
where the $1$ and $-1$ are paired and the $1$ is a defect in
$\alpha = \alpha' || (1)|| \alpha''$. As before, $\alpha'$ and
$\alpha''$ are $1$-factors and $i_{\alpha''} = 0$. There are two
subcases to consider: either
\[
\beta^+ = \beta' || (-1) || \beta'' || (1) \quad\text{or}\quad
\beta^+ = \beta' || (1) || \beta'' || (1) 
\]
where the lengths of $\alpha'$, $\beta'$ agree (thus similarly for the
lengths of $\alpha''$, $\beta''$). 

\textbf{Subcase A.} Suppose that $\beta^+ = \beta' || (-1) || \beta''
|| (1)$. Then $\bil{\nat(\alpha^{(j)})}{\omax(\beta)} = 0$ for
all $j < d-1$ by the same argument as before, by Lemma
\ref{l:defect}. Also, there must be a pairing in $\alpha'$ or
$\alpha''$ with a corresponding incompatibility in $\beta$ (meaning
that the corresponding entries in $\beta$ have the same sign). As
$\alpha^{(d-1)} = \alpha' ||(-1)|| \alpha''$, that same
incompatibility makes $\beta = \beta' || (-1) || \beta''$ must be
$\alpha^{(j-1)}$-incompatible, so the right hand side of the equality
\[
\bil{\nat(\alpha^-)}{\omax(\beta^+)} = - \sum_{j=1}^{d-1}
\frac{[j]}{[d]} \bil{\nat(\alpha^{(j)})}{\omax(\beta)}
\]
is equal to zero, proving the desired result in this subcase.

\textbf{Subcase B.} 
Suppose $\alpha^-= \alpha'||(1)||\alpha ''||(-1)$ and $\beta^+ =
\beta'||(1)||\beta''||(1)$. In this case we already have an
incompatibility in $\beta^+$ with the rightmost pairing in
$\alpha^-$. If there is another such incompatibility in $\beta^+$ then
necessarily this incompatibility occurs with respect to a pairing in
either $\alpha'$ or $\alpha''$. This second incompatibility remains an
$\alpha^{(j)}$-incompatibility in $\beta$ for all $j=1,\ldots
d-1$. Thus $\bil{\nat(\alpha^-)}{\omax(\beta^+)} = 0$, as
desired.

It remains only to consider the case where there are no incompatible
pairings in $\alpha^-$ except the rightmost one. By weight
considerations, there must be exactly one defect in $\alpha^-$ with
corresponding entry in $\beta^+$ equal to $-1$. This defect is
necessarily in $\alpha'$. Suppose this is the $k$th defect in
$\alpha$. Then this defect remains a defect in all $\alpha^{(j)}$
except when $j=k-1$ or $j=k$, so for those values of $j$ we have
$\bil{\nat(\alpha^{(j)})}{\omax(\beta)} = 0$. Therefore
\[
\bil{\nat(\alpha^-)}{\omax(\beta^+)} =
\frac{-1}{[d]}\left( [k-1]\bil{\nat(\alpha^{(k-1)})}{\omax(\beta)}
+ [k]\bil{\nat(\alpha^{(k)})}{\omax(\beta)} \right).
\]
{\bf Claim:} $\beta$ is compatible with both $\alpha^{(k-1)}$ and
$\alpha ^{(k)}$. We show this first for $\beta$ and
$\alpha^{(k)}$. Consider a pairing in $\alpha^{(k)}$. If this pairing
is not the one that links the $k$th defect in $\alpha$ with the next
one, then the corresponding signs in $\beta$ must be different. For if
these signs were the same, then this pairing would also be a pairing
in $\alpha$ and with corresponding signs $\beta$ the same, which in
turn represents an incompatible pairing in $\alpha^-$ different from
the rightmost one, contradicting the hypothesis in this case. Now
consider the pairing which links the $k$th defect in $\alpha$ with
the next one. The signs in the corresponding entries of $\beta$ cannot
be the same. For if they were, they would both be equal to $-1$ and
would force another pairing in $\alpha ^{(k)}$ with corresponding
signs in $\beta$ both 1. This forced pairing in $\alpha^{(k)}$ is also
an incompatible pairing in $\alpha^-$ different from the rightmost
one, again contradicting the hypothesis. Thus $\beta$ is
$\alpha^{(k)}$-compatible. A similar argument verifies that $\beta$ is
$\alpha^{(k-1)}$-compatible and the claim is proved.

Since $\alpha^{(k-1)}, \alpha^{(k)}$, and $\beta$ have shorter lengths
than $\alpha^-$ and $\beta^+$, we have
\begin{eqnarray*}
  \bil{\nat(\alpha^{(k-1)})}{\omax(\beta)} &=& v^m
  \prod_{j=1}^m[s_j'(\alpha^{(k-1)},\beta)]\\ \bil{
  \nat(\alpha^{(k)})}{\omax(\beta)} &=& v^m
  \prod_{j=1}^m[s_j'(\alpha^{(k)},\beta)]
\end{eqnarray*}
by the inductive hypothesis. For $j\neq k$ we have
$s'_j(\alpha^{(k-1)},\beta)=s'_j(\alpha^{(k)},\beta)$ because the
entries in $\beta$ corresponding to defects $k-1$ and $k$ of $\alpha$
must be 1, since by hypothesis the unique defect of $\alpha$ with a
$-1$ corresponding entry in $\alpha$ is the $k$th defect. Since
$\beta$ is compatible with $\alpha^{(k-1)}$ and $\alpha^{(k)}$, we
have $s'_j(\alpha^{(k-1)},\beta)=k$ and
$s'_j(\alpha^{(k)},\beta)=-(k-1).$ Therefore
\begin{multline*}
  [k-1]\bil{\nat(\alpha^{(k-1)})}{\omax(\beta)} + [k]
  \bil{\nat(\alpha^{(k)})}{\omax(\beta)} \\ = v^m
  \left(\prod _{j\neq k}
       [s_j'(\alpha^{(k)},\beta)]\right) \big([k-1][k]-[k][k-1] \big)
       =0
\end{multline*}
which leads to the desired result that
$\bil{\nat(\alpha^-)}{\omax(\beta^+)} = 0$.
\end{proof}

\begin{cor}\label{c:closed-formula}
Under the same hypotheses as the previous theorem, we have
\[
\pi_{\alpha,\beta} = \frac{\bil{\nat(\alpha)}{\omax(\beta)}}%
{\bil{\omax(\beta)}{\omax(\beta)}} =
  \begin{cases}
  \prod_{j=1}^m \frac{[s'_j]}{[s_j][s_j+1]} & \text{ if $\beta$ is
    $\alpha$-compatible} \\
  \qquad 0  & \text{ otherwise.}
  \end{cases}
\]
In the nonzero case, the above product simplifies to the reciprocal of
a product of quantum integers, up to sign.
\end{cor}

\begin{proof}
This is immediate from the preceding theorem combined with formula
\eqref{e:length-squared}. The last claim follows from the definition
of $s'_j(\alpha,\beta)$, which always (up to sign) cancels with a
factor in the denominator.
\end{proof}

\begin{example}
If $\alpha = (1^{m+1},-1)$ then $\nat(\alpha) = \frac{1}{[m]}
\omax(\alpha) - \sum_{j=1}^{m-1} \frac{[j]}{[d]} \Phi_1
\nat(\alpha^{(j)})$, by Corollary~\ref{c:recursion}(b). There is
only one $1$-factor $\beta \ne \alpha$ which is $\alpha$-compatible,
namely $\beta = (1^m,-1,1)$. By Corollary \ref{c:closed-formula}, we
have
\[
\nat(\alpha) = \frac{1}{[m]} \omax(\alpha) - \frac{1}{[m]}
\omax(\beta).
\]
This indicates that there is cancellation when one continues to
recursively expand the right hand side of $\nat(\alpha) =
\frac{1}{[m]} \omax(\alpha) - \sum_{j=1}^{m-1} \frac{[j]}{[d]}
\Phi_1 \nat(\alpha^{(j)})$ using Corollary~\ref{c:recursion}(b).
\end{example}

\section{Cellularity}\label{s:cellular}\noindent
In this section we prove that the basis $\{\nat(\alpha)\}$ indexed by
$1$-factors $\alpha$ of shape $\lambda$ (where $\lambda$ is a
partition of $n$ with at most two parts) form a cellular basis of the
cell module $C(\lambda)$.  Axiomatics for cellular algebras were
introduced in \cite{GL:96}, and the Temperley--Lieb algebra
$\TL_n(\delta)$ was a primary motivating example.

For the moment, let $\Bbbk$ be a commutative ring.  Consider the
diagram algebra which has $\Bbbk$-basis consisting the planar Brauer
diagrams: graphs on $2n$ vertices (arranged in two parallel rows of
$n$ each) with $n$ edges, such that each vertex is joined to exactly
one other and the edges do not intersect when drawn in the rectangle
determined by the vertices. If $D_1$, $D_2$ are two such diagrams then
their composite diagram $D_1 \circ D_2$ is formed by stacking $D_1$
above $D_2$ and removing interior loops and vertices (this is another
basis diagram). The (associative) multiplication in the algebra is
given by the rule
\[
D_1 D_2 = \delta^{N} (D_1 \circ D_2)
\]
where $N$ is the number of removed interior loops. 

\begin{thm}[Kauffman]
The Temperley--Lieb algebra $\TL_n(\delta)$ is isomorphic to the
diagram algebra described in the preceding paragraph.
The isomorphism takes the generator $e_i$ to the diagram
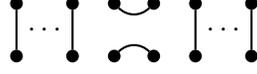
\[
\onepic \cdots \onepic\quad \epic \quad\onepic \cdots \onepic
\]
in which the $i$th vertex in each row is horizontally linked to the
next one.
\end{thm}

Let $L_0$ be the span of the set of link diagrams (introduced in the
paragraph before Example~\ref{ex:link-dia}) on $n$ vertices. Define an
action of $\TL_n(\delta)$ on $L_0$ by letting a diagram $D$ act by
``diagram multiplication'' as follows. Stack $D$ above a given link
diagram $\ell$ and remove any loops. The resulting composite diagram
is a new link diagram $D \circ \ell$. Then we define
\[
D \ell = \delta^{N} (D \circ \ell) 
\]
where $N$ is the number of loops that were removed. With this action,
$L_0$ becomes a $\TL_n(\delta)$-module. Furthermore, the link diagram
$D \circ \ell$ always has at least as many links as does $\ell$ (and
it can have more in some cases). For any $d \in X(n)_+$, let $L_0^{\le
  d}$ (resp., $L_0^{< d}$) be the $\Bbbk$-submodule spanned by the set
of link diagrams with at most (resp., fewer than) $d$ defects, and define
\[
L(d) := L_0^{\le d} / L_0^{< d}.
\]
We will also write $L(\lambda) = L(d)$, where $\lambda$ is the
(unique) partition of at most two parts such that $\lambda_1 -
\lambda_2 = d$.  If $\delta \ne 0$, the set $\{ L(\lambda) \}$
(indexed by such partitions) is a complete set of cell modules for
$\TL_n(\delta)$. In case $\Bbbk$ is a field and $\TL_n(\delta)$ is
semisimple, the cell modules are all simple modules, and we get a
complete set of isomorphism classes of simples.

Now we let $\alpha^\ell$ be the $1$-factor corresponding to a link
diagram $\ell$ under the bijection in Lemma \ref{l:bijections}. Define
a map $\varphi: L(\lambda) \to C(\lambda)$ by the rule
\begin{equation}\label{e:cell-iso}
\varphi(\ell) = \nat(\alpha^\ell)  .
\end{equation}
The following is the main result of this section.

\begin{thm}
Let $\Bbbk$ be a field.  Assume that $[n]! \ne 0$, where $0 \ne v \in
\Bbbk$, and set $\delta = v+v^{-1}$. Then the map $\varphi$ defined in
\eqref{e:cell-iso} is an isomorphism $L(\lambda) \cong C(\lambda)$ of
$\TL_n(-\delta)$-modules. We also have $L(\lambda) \cong C(\lambda)$,
as $\TL_n(\delta)$-modules.
\end{thm}

\begin{proof}
As $\varphi$ is a bijection, we only need to show that it is a
morphism of $\TL_n(-\delta)$-modules.  Since $\TL_n(-\delta)$ is
generated by the $e_i$ ($i = 1, \dots, n-1$) it suffices to show that
$e_i \varphi(\ell) = \varphi(e_i \ell)$ for each $i$.

Fix $\ell$, and set $\alpha := \alpha^\ell$, $\beta := \alpha^{e_i
  \ell}$.  Let $\lambda$ be the shape of $\alpha$. The proof breaks
into seven cases based on the value of the pair $(\alpha_i,
\alpha_{i+1})$, as enumerated below:
\begin{enumerate}\renewcommand{\labelenumi}{\arabic{enumi}.}
\item $(\alpha_i, \alpha_{i+1}) = (1,1)$ and both values are defects
  in $\alpha$.
\item $(\alpha_i, \alpha_{i+1}) = (1,1)$ and only one (necessarily the
  left) is a defect in $\alpha$.
\item $(\alpha_i, \alpha_{i+1}) = (1,1)$ and neither is a defect in
  $\alpha$.
\item $(\alpha_i, \alpha_{i+1}) = (-1,1)$ and the second term is a defect in
  $\alpha$.
\item $(\alpha_i, \alpha_{i+1}) = (-1,1)$ and the second term is not a
  defect in $\alpha$.
\item $(\alpha_i, \alpha_{i+1}) = (1,-1)$.
\item $(\alpha_i, \alpha_{i+1}) = (-1,-1)$.
\end{enumerate}
One needs to consider the cases one by one, using the explicit sum formula
\[
\nat(\alpha) = \sum_{\gamma \in \mathcal{S}(\alpha)}
(-v)^{\sigma(\alpha,\gamma)} y_{\gamma_1, \dots, \gamma_n}
\]
from Lemma~\ref{l:N-formula} and the action of the $e_i$ on
$V^{\otimes n}$ defined in equation \eqref{e_i-action}.

\textbf{Case 1.} In this case, $e_i \ell$ has one more link than
$\ell$, so $e_i \ell = 0$ in the quotient module
$L(\lambda)$. Furthermore, each term $y_\gamma = y_{\gamma_1, \dots,
  \gamma_n}$ on the right hand side of the explicit sum formula for
$\nat(\alpha)$ has $y_{1,1}$ in tensor positions $i$, $i+1$. By
\eqref{e_i-action} we have $e_i \nat(\alpha) = 0$, so this case is
done.

\textbf{Case 2.} In this case the rightmost $1$ in $(\alpha_i,
\alpha_{i+1}) = (1,1)$ is linked to some $-1$ to its right in
$\alpha$. Let $j>i+1$ be the position of that $-1$. The $1$-factor
$\beta$ is obtained from $\alpha$ by interchanging the terms in those
positions. Consider the right hand side of the sum formula for
$\nat(\alpha)$. Since $\alpha$ has a defect in position $i$, every
$y_\gamma$ has $\gamma_i = 1$. Hence we see only terms on that right
hand side containing factors in tensor positions $i$, $i+1$ of one of
the two forms $y_{\gamma_i,\gamma_{i+1}} = y_{1,1}$ or $y_{1,-1}$. The
action of $e_i$ kills all of the former terms, and acts as
$-v^{-1}y_{1,-1} + y_{-1,1}$ in those positions. After making that
replacement in each of the surviving terms, we obtain the sum formula
for $\nat(\beta)$. Case 2 is proved, but to make the argument
concrete, we give an example. Suppose that $\alpha =
(1^3,-1^2)$ and $i=1$. The second $1$ in $\alpha$ is linked to the
last $-1$, so $j = 5$ and $\beta = (1,-1,1,-1,1)$. Now
\begin{align*}
\nat(\alpha) &= y_{1,1,1,-1,-1} - v y_{1,1,-1,1,-1} - v
y_{1,-1,1,-1,1} + v^2 y_{1,-1,-1,1,1} \intertext{and $e_1$ kills the
  first two terms, replacing the initial factor of $y_{1,-1}$ in the
  surviving terms by $-v^{-1}y_{1,-1} + y_{-1,1}$, to produce}
\nat(\beta) &= y_{1,-1,1,-1,1} - v y_{-1,1,1,-1,1} - v
y_{1,-1,-1,1,1} + v^2 y_{-1,1,-1,1,1}
\end{align*}
as required. 

\textbf{Case 6.} In this case $\alpha$ has a pairing in tensor
positions $i$, $i+1$ and $\beta = \alpha$. Carrying out the
calculation in general is no more difficult than doing it for $i=1$
and $\alpha=(1,-1)$. Then $\nat(\alpha) = z_0 = y_{1,-1}-vy_{-1,1}$
and from equation \eqref{e_i-action} we see that
\begin{align*}
e_1 \nat(1,-1) &= (-v^{-1}y_{1,-1}+y_{-1,1}) - v(y_{1,-1}-vy_{-1,1}) \\
&= -(v+v^{-1}) \nat(1,-1).
\end{align*}

We leave the remaining cases, which are similar, to the reader. The
exercise of working through an illustrative example for each remaining
case is quite instructive.

Because of the minus sign appearing in the Case 6 calculation,
$\varphi$ is an isomorphism of $\TL_n(-\delta)$-modules, where
$\delta=v+v^{-1}$. The final claim follows because $\TL_n(\delta)
\cong \TL_n(-\delta)$.
\end{proof}

\begin{rmk}\label{r:signs}
Suppose one changes the sign choice in equation \eqref{e_i-action}.
In other words, suppose that $\dot{e}_i$ is defined by
\[
\dot{e}_i = 1^{\otimes(i-1)} \otimes (v+v^{-1})\dot{z}_0
\otimes 1^{\otimes (n-1-i)}.
\]
Then the map $\varphi$ is no longer a $\TL_n(-\delta)$-isomorphism,
nor is it a $\TL_n(\delta)$-isomorphism (it is enough to look at the
case $\lambda=(2,1)$ to see the issue).  This justifies our choice of
sign in \eqref{e_i-action}.
\end{rmk}

\section{Action of $\TL_n$ on the orthogonal maximal vectors}%
\label{s:action}\noindent
Let $\Bbbk$ be a field, and set $\delta = \pm(v+v^{-1})$.  We
now derive an explicit formula for the action of $\TL_n(\delta)$ on
the $\omax(\alpha)$, for a given $1$-factor $\alpha$ of length $n$.
Given $i \ge 1$ which is strictly less than the length of $\alpha$,
write
\[
\alpha = \alpha'||(\alpha_i,\alpha_{i+1})||\alpha''
\]
and let $w = w_i(\alpha)$ be the sum of the entries in $\alpha'$. Then
of course $w \ge 0$. By Lemma~\ref{l:tensor-action}, equation
\eqref{e_i-action} defines a $\TL_n(\delta)$-module structure on
$V^{\otimes n}$.

\begin{prop}\label{p:action}
Fix $0 \ne v \in \Bbbk$ satisfying $[n]! \ne 0$.  For any $1$-factor
$\alpha$ of length $n$ such that $\alpha_i \alpha_{i+1} = -1$, set
$\beta = \alpha'||(-\alpha_i,-\alpha_{i+1})||\alpha''$.
\begin{enumerate}
\item If $w = w_i(\alpha) = 0$ then $e_i \omax(\alpha) = -(v+v^{-1})
  \omax(\alpha)$.
\item If $w = w_i(\alpha) > 0$ then
  \[
  e_i \omax(\alpha) =
  \begin{cases}
    \frac{[w+2]}{[w+1]} \big(\omax(\beta) - \omax(\alpha)\big) &
      \text{if } (\alpha_i, \alpha_{i+1}) = (1,-1) \\
    \frac{[w]}{[w+1]} \big((\omax(\beta) - \omax(\alpha)\big) &
    \text{if } (\alpha_i, \alpha_{i+1}) = (-1,1) \\
    \qquad\qquad 0 & \text{otherwise.}
  \end{cases}
  \]
\end{enumerate}
\end{prop}

\begin{proof}
The proof is by induction on $n$. The base case
$n=2$ is clear.

{\bf Case 1}: Suppose $i<n-1$ and that $\alpha$ is a $1$-factor of
length $n-1$. We need to verify the formulas for $e_{i}
\omax(\alpha^+)$ and $e_{i} \omax(\alpha^-)$.

First consider $e_i \omax(\alpha^+)$.  By induction the formula for
the expansion of $e_i \omax(\alpha)$ can be applied in the equality
$e_i\omax(\alpha^+) = (e_i \omax(\alpha)) \otimes y_1$. The result
then follows in all cases by direct calculation.

Next consider $e_i \omax(\alpha^-)$, still under the assumption that
$i<n-1$. Then by definition $e_i \omax(\alpha^-) =
e_i(\Phi_2(\omax(\alpha)))$, so
\begin{equation*}
  e_i \omax(\alpha^-) = [i_{\alpha}] (e_i \omax(\alpha))\otimes y_{-1}
  - v^{i_{\alpha}} e_i(F\omax(\alpha))\otimes y_1.
\end{equation*}
If $\alpha_i \alpha_{i+1} = 1$ then both terms on the right hand side
above are zero by induction, so we are done in that case.  Suppose next
that $(\alpha_{i},\alpha_{i+1})=(1,-1)$ with $w_i(\alpha)>0$ and set
$c=\frac{[w_i(\alpha)+2]}{[w_i(\alpha)+1]}$. Then by induction we have
\begin{align*}
e_i \omax(\alpha^-) &= c\,[i_{\alpha}] \left(\omax(\alpha'||
(-1,1)||\alpha '')-\omax(\alpha'|| (1,-1)||\alpha '')\right) \otimes
y_{-1}\\ &\qquad -c\,v^{i_{\alpha}} \left(F(\omax(\alpha'||
(-1,1)||\alpha ''))-F(\omax(\alpha'|| (1,-1)||\alpha ''))\right)
\otimes y_1\\ &= c \left(\Phi_2(\omax(\alpha'|| (-1,1)||\alpha ''))
-\Phi_2(\omax(\alpha'|| (1,-1)||\alpha ''))\right)\\ &=
c\left(\omax(\alpha'|| (-1,1)||\alpha ''||(-1)) - \omax(\alpha'||
(1,-1)||\alpha ''||(-1))\right)
\end{align*}
and the formula is verified in this case. The verification for the
remaining case $(\alpha_{i},\alpha_{i+1})=(- 1,1)$ is similar and
omitted.

{\bf Case 2}: Let $i=n-1$ and suppose $\alpha$ is a $1$-factor of
length $n-2$. We consider $\alpha ||(\alpha_{n-1}, \alpha_n)$ where
the signs of $\alpha_{n-1}$ and $\alpha_n$ are the same or differ.  If
the product $\alpha_{n-1} \alpha_n = 1$ (the signs are the same) then
$e_{n-1} \omax(\alpha||(\alpha_{n-1}, \alpha_n)) = 0$, and we are
done.

It remains only to deal with the case where the product
$\alpha_{n-1}\alpha_n = -1$ (the signs are opposite). Set
$\alpha^{+-}=\alpha ||(1,-1)$ and $\alpha^{-+}=\alpha ||(-1,1)$. An
elementary (tedious) calculation with the definitions yields
\[
\omax(\alpha^{-+})-\omax(\alpha^{+-})=-[i_{\alpha}+1]
\omax(\alpha)\otimes\left( y_{1,-1} - v y_{-1,1} \right).
\]
Next consider the expansions of $e_{n-1} \omax(\alpha^{+-})$,
$e_{n-1} \omax(\alpha^{-+})$. We have
\begin{align*}
  e_{n-1} \omax(\alpha^{+-}) &= -[i_{\alpha}+2]\omax(\alpha)\otimes \left(
  y_{1,-1}-vy_{-1,1} \right)\\ \intertext{and} e_{n-1} \omax(\alpha^{-+})
  &= [i_{\alpha}]\omax(\alpha)\otimes \left( y_{1,-1}-vy_{-1,1}
  \right).
\end{align*}
The above calculations show that
\begin{align*}
e_{n-1} \omax(\alpha^{+-}) &=
\frac{[i_{\alpha}+2]}{[i_{\alpha}+1]}
\left(\omax(\alpha^{-+})-\omax(\alpha^{+-})\right) \\
\intertext{and}
e_{n-1} \omax(\alpha^{-+}) &= \frac{[i_{\alpha}]}{[i_{\alpha}+1]}
\left(\omax(\alpha^{+-})-\omax(\alpha^{-+})\right).
\end{align*}
For $w_{n-1}(\alpha) > 0$, we have
$w_{n-1}(\alpha^{+-})=w_{n-1}(\alpha^{-+})=i_{\alpha}$, so the above expressions
verify the desired formulas. The proof is complete.
\end{proof}

%%%%%%%%%%%%%%%%%%%%%%%%%%%%%%%%%%%%%%%%%%%%%%%%%%%%%%%%%%%%%
% bibliography using amsrefs package 
%%%%%%%%%%%%%%%%%%%%%%%%%%%%%%%%%%%%%%%%%%%%%%%%%%%%%%%%%%%%%

\end{document}